\documentclass[11pt, oneside]{amsart}

\pdfoutput=1
\usepackage[centering]{geometry}                		
\usepackage{graphicx}				

\usepackage{amssymb}
\usepackage{bm}
\usepackage{tikz-cd}

\theoremstyle{definition}
\newtheorem{definition}{Definition}
\newtheorem{remark}{Remark}

\theoremstyle{plain}
\newtheorem{theorem}{Theorem}
\newtheorem{corollary}{Corollary}
\newtheorem{lemma}{Lemma}
\newtheorem{proposition}{Proposition}

\newcommand\ord{\mathop{\mathrm{ord}}\nolimits}
\newcommand\exmid{\mathrel{\|}}

\newcommand\Kl{\mathrm{Kl}}
\newcommand\JJ{\mathcal{J}}

\newcommand\dd{\mathrm{d}}

\numberwithin{equation}{section}

\usepackage{xcolor}
\definecolor{couleur_cite}{rgb}{0.05,.4,0.05}
\definecolor{couleur_link}{rgb}{0.05,0.05,0.4}
\definecolor{couleur_url}{rgb}{0.5,0,0}
\usepackage[hyperfootnotes=false]{hyperref}
\hypersetup{hypertexnames=false,colorlinks=true,citecolor=couleur_cite,linkcolor=couleur_link,urlcolor=couleur_url,
pdfstartview=FitH, pdfauthor=Djordje Milicevic and Sichen Zhang, pdftitle=Distribution of Kloosterman paths to high prime power moduli}

\title[Distribution of Kloosterman paths to high prime power moduli]{Distribution of Kloosterman paths\\ to high prime power moduli}
\author{Djordje Mili\'cevi\'c, Sichen Zhang}

\address{Bryn Mawr College, Department of Mathematics, 101 North Merion Avenue, Bryn Mawr, PA 19010, USA}
\curraddr{Max-Planck-Institut f\"ur Mathematik, Vivatsgasse 7, D-53111 Bonn, Germany}
\email{dmilicevic@brynmawr.edu}

\address{Bryn Mawr College, Department of Mathematics, 101 North Merion Avenue, Bryn Mawr, PA 19010, USA}
\email{czhang2@brynmawr.edu}

\thanks{D.M. was supported by National Science Foundation Grant DMS-1903301.}

\begin{document}

\begin{abstract}
We consider the distribution of polygonal paths joining the partial sums of normalized Kloosterman sums modulo an increasingly high power $p^n$ of a fixed odd prime $p$, a pure depth-aspect analogue of theorems of Kowalski--Sawin and Ricotta--Royer--Shparlinski. We find that this collection of Kloosterman paths naturally splits into finitely many disjoint ensembles, each of which converges in law as $n\to\infty$ to a distinct complex valued random continuous function. We further find that the random series resulting from gluing together these limits for every $p$ converges in law as $p\to\infty$, and that paths joining partial Kloosterman sums acquire a different and universal limiting shape after a modest rearrangement of terms. As the key arithmetic input we prove, using the $p$-adic method of stationary phase including highly singular cases, that complete sums of products of arbitrarily many Kloosterman sums to high prime power moduli exhibit either power savings or power alignment in shifts of arguments.
\end{abstract}

\maketitle

\section{Introduction}

This paper is inspired by the beautiful work of Kowalski--Sawin~\cite{KowalskiSawin2016} and Ricotta, Royer, and Shparlinski~\cite{RicottaRoyer2018,RicottaRoyerShparlinski2020}, which considered the distribution of Kloosterman paths, polygonal paths joining the partial sums of Kloosterman sums modulo a large prime $p$, and a fixed  $n^{\text{th}}$ power $p^n$ of a large prime $p$, respectively. Paths traced by incomplete exponential sums such as those in Figures~\ref{p-3-figure} and \ref{larger-prime-figure} give a fascinating insight into the chaotic formation of the square-root cancellation in the corresponding complete sums and have been studied since at least \cite{Lehmer1976,Loxton1983,Loxton1985}. In this paper, we address the pure depth-aspect distribution of Kloosterman paths to moduli that are high powers $p^n$ of a fixed odd prime $p$ and find many new, surprisingly distinctive features.

\subsection{Kloosterman paths}
\label{paths-intro-sec}
Let $p$ be an odd prime and $n\in\mathbb{N}$. For $a,b\in(\mathbb{Z}/p^n\mathbb{Z})^{\times}$, we define the normalized Kloosterman sum as
\begin{equation}
\label{kloost-def}
\Kl_{p^n}(a,b)=\frac{1}{p^{n/2}}S(a,b;p^n)=\frac{1}{p^{n/2}}\sum_{1\le x\le p^n,(x,p)=1}e_{p^n}(ax+b\overline{x}),
\end{equation}
where $e_{p^n}(z)=e(z/p^n)=e^{2\pi iz/p^n}$, and $\overline{x}$ denotes the inverse of $x$ modulo $p^n$. The normalization reflects the square-root cancellation in these complete exponential sums consisting of $\varphi(p^n)$ terms:
\[ |\Kl_{p^n}(a,b)|\leqslant 2, \]
a result of Weil's deep algebro-geometric bound~\cite{Weil1948} for $n=1$ and of the explicit evaluation using the $p$-adic method of stationary phase for $n\geqslant 2$ (which we collect in Lemma~\ref{kloost-eval}). In particular, in this latter case, $\Kl_{p^n}(a,b)=0$ unless $ab\in(\mathbb{Z}/p^n\mathbb{Z})^{\times 2}$.

By the time the summation in \eqref{kloost-def} is completed, substantial cancellation has occurred and one lands at an end point in the real interval $[-2,2]$. In a reasonable ensemble (varying some subset of $a$, $b$, $p$, $n$), its limiting distribution can often be fruitfully described by an appropriate Sato--Tate measure; more on that below. Kloosterman paths describe \emph{the road taken} to that end point. 

Specifically, for $j\in\JJ_{p^n}:=\{j\in\{1,...,p^{n}\}, p\nmid j\}=\{j_1<\dots<j_{\varphi(p^{n})}\}$, define the partial sum
\[ \Kl_{j;p^n}(a,b)=\frac{1}{p^{n/2}}\sum_{1\le x\le j, p\nmid x}e\left(\frac{ax+b\bar{x}}{p^n}\right), \]
and define the \emph{Kloosterman path} as the polygonal path $\gamma_{p_n}(a,b)$ obtained by concatenating the closed segments $[\text{Kl}_{j_i;p^n}(a,b_0),\text{Kl}_{j_{i+1};p^n}(a,b_0)]$ for all $j_i\in\JJ_{p^n}$.  We may also think of this path as a continuous map $\Kl_{p^n}(\cdot;(a,b)):[0,1]\to\mathbb{C}$, $t \mapsto \text{Kl}_{p^n}(t;(a,b))$ obtained by parameterizing the path $\gamma_{p^n}(a,b)$ so that each of the $\varphi(p^n)-1$ segments is parametrized linearly on an interval of length $1/(\varphi(p^n)-1)$.

For a fixed $b_0\in\mathbb{Z}$ with $p\nmid b_0$, the map from $(\mathbb Z/p^n\mathbb Z)^{\times}\to C^0([0,1],\mathbb C)$ given by
\begin{equation}
\label{kl-pn-def}
a\mapsto \Kl_{p^n}(t;(a,b_0))
\end{equation}
can be viewed as a random variable on the finite probability space $(\mathbb Z/p^n\mathbb Z)^\times$ with the uniform probability measure with values in the space of complex valued continuous functions on [0,1] endowed with the supremum norm, and we denote this random variable by $\Kl_{p^n}(t)$.

Kowalski--Sawin~\cite[Theorem 1.1]{KowalskiSawin2016} show that, in the prime case $n=1$, the random variable $\Kl_p$ converges as $p\to\infty$, in the sense of finite distributions, to a specific random Fourier series (that is, a $C^0([0,1],\mathbb{C})$-valued random variable). Ricotta--Royer~\cite[Theorem A]{RicottaRoyer2018} prove the corresponding theorem for a fixed $n\geqslant 2$ and $p\to\infty$, indentifying a \emph{different} limiting random Fourier series, and Ricotta--Royer--Shparlinski~\cite{RicottaRoyerShparlinski2020} prove that (for $n\geqslant 31$) this latter convergence holds in law, a substantial stregthening. We refer the reader to \S\ref{prob-sec} for a summary of probabilistic notions and tools.

\subsection{Main result}
In this paper, we address the pure depth-aspect analogue of Kowalski--Sawin~\cite{KowalskiSawin2016}, the question of distribution of Kloosterman paths modulo $p^n$ for a fixed prime $p$ and $n\to\infty$. Ricotta, Royer, and Shparlinski note \cite[p.176]{RicottaRoyerShparlinski2020}, \cite[p.498]{RicottaRoyer2018} that ``this problem, both theoretically and numerically, seems to be of completely different nature''.

\begin{figure}
\centering
\includegraphics[width=\textwidth]{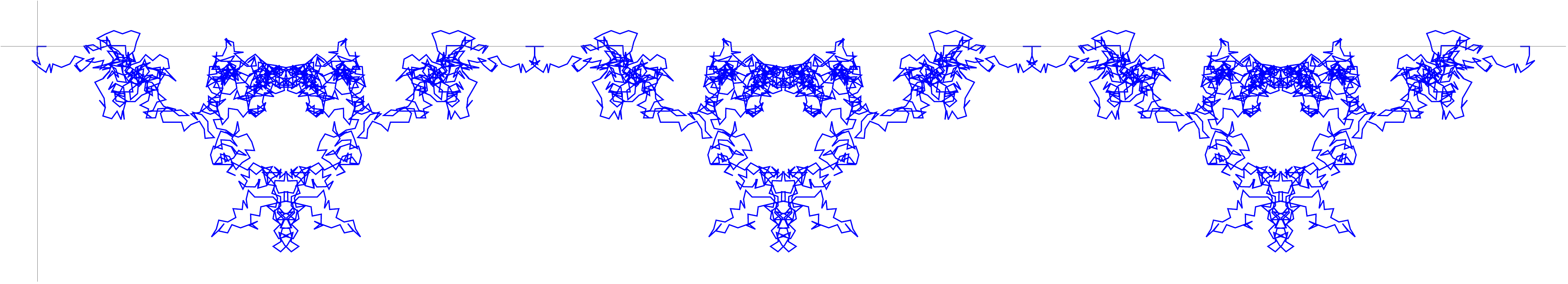}
\includegraphics[width=0.43\textwidth]{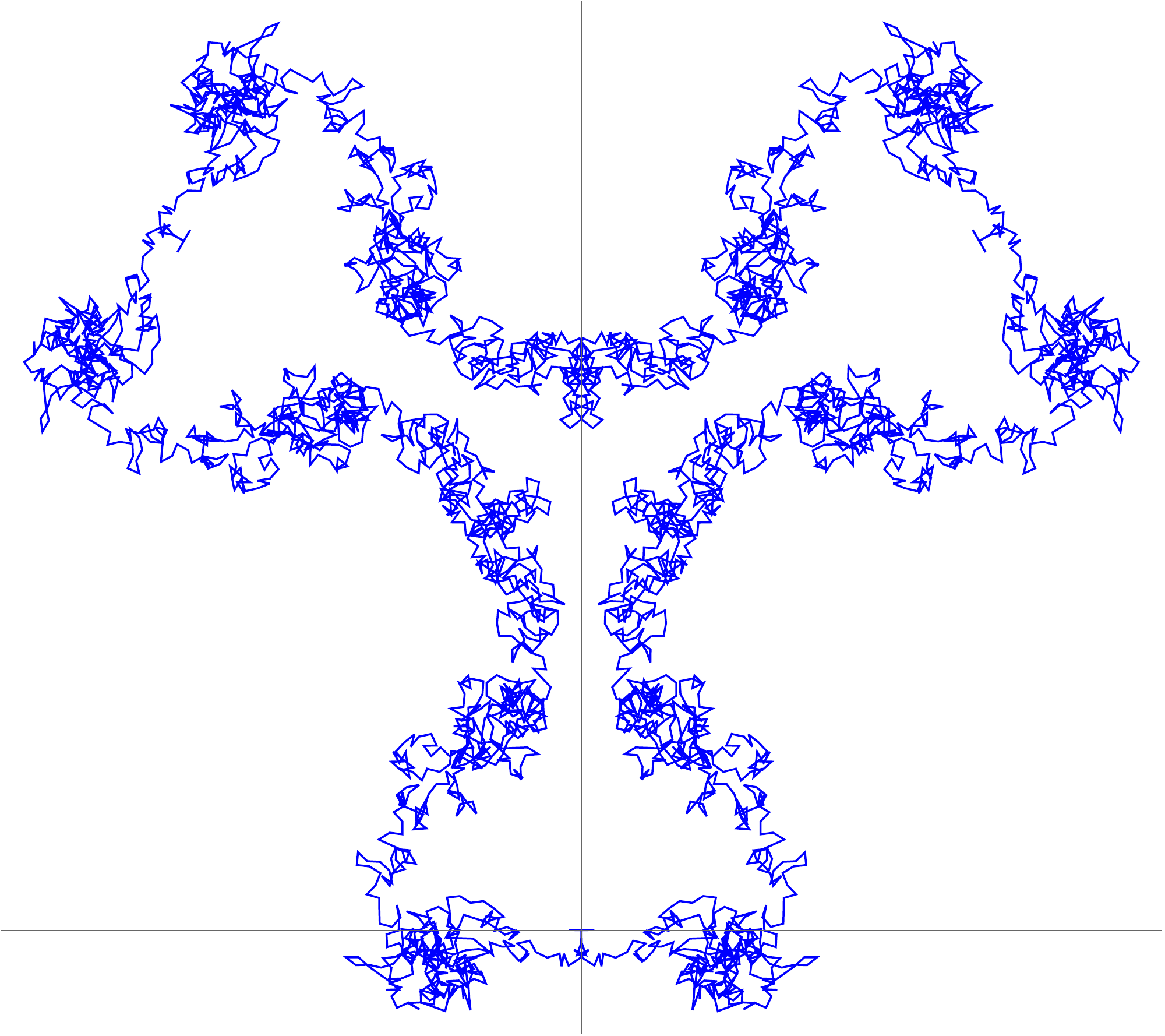}
\caption{Kloosterman paths $\Kl_{3^8}(t;(1,1))$ and $\Kl_{3^8}(t;(5,1))$.}
\label{p-3-figure}
\end{figure}

\begin{figure}
\centering
\includegraphics[width=0.9\textwidth]{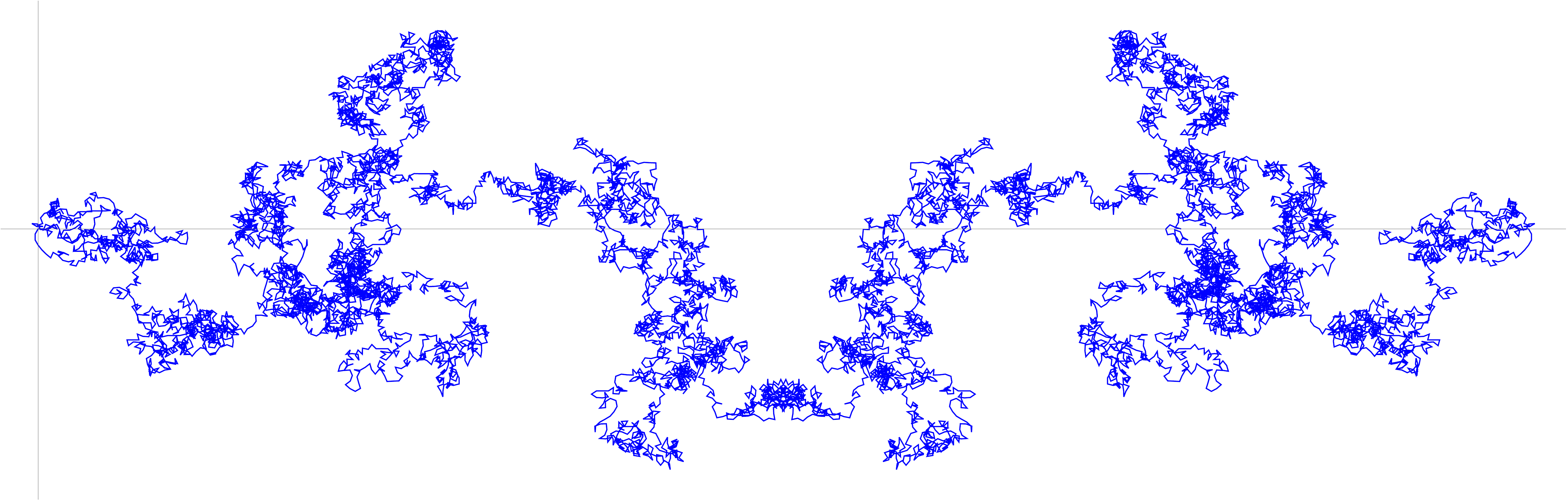}
\includegraphics[width=0.43\textwidth]{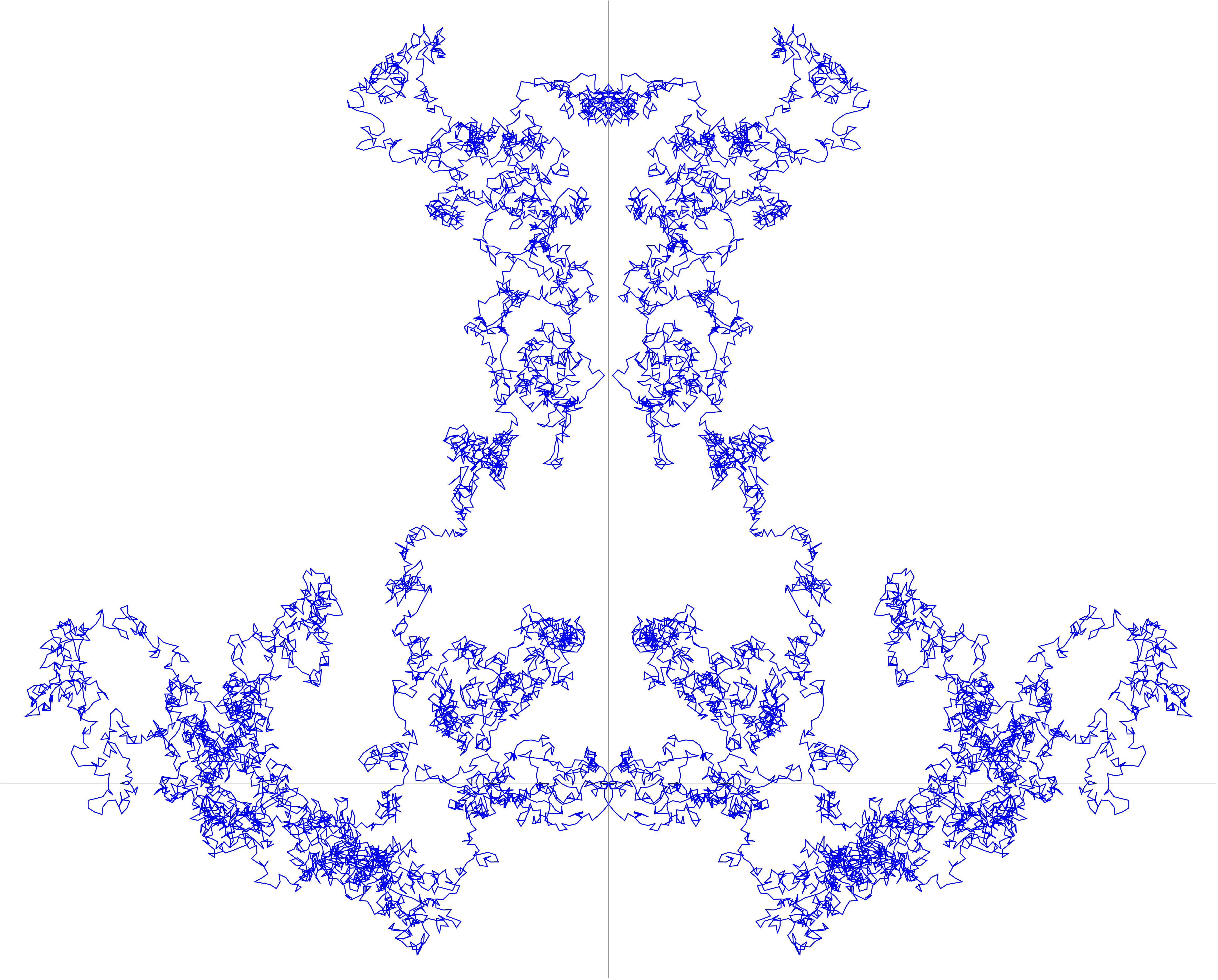}
\caption{Kloosterman paths $\Kl_{5^6}(t;(1,1))$ and $\Kl_{5^6}(t;(2,1))$.}
\label{larger-prime-figure}
\end{figure}

Indeed, one needs to look no further than a couple of sample paths modulo large $3^n$ to see that, say, paths $\Kl_{3^n}(a,1)$ with $a\equiv 1\pmod 3$ and $a\equiv 2\pmod 3$ look nothing like each other; see Figure~\ref{p-3-figure}. In particular, they exhibit a stark translational and rotational symmetry of order 3, respectively. Much of the invariance distinction becomes less visually obvious for larger primes $p$, though, with all but the obvious bilateral symmetry seemingly disappearing as in Figure~\ref{larger-prime-figure}, and one might start to wonder if the case $p=3$ is a fluke. On the other hand, about half of the paths tend to wander away in the horizontal direction while the rest do not, which of course corresponds to the distribution of the complete sum $\Kl_{p^n}(a,b)$ and its vanishing when $ab\not\in(\mathbb{Z}/p\mathbb{Z})^{\times 2}$.

We will see that the distinction among different classes $[a\bmod p]\in(\mathbb{Z}/p\mathbb{Z})^{\times}$, which is a well-defined distinction in the family with a fixed $p$, persists and it becomes harder, and indeed more artificial, to describe the joint distribution of all paths as $a\in(\mathbb{Z}/p^n\mathbb{Z})^{\times}$. Instead, for $n\geqslant 2$ it becomes natural to split the ensemble of Kloosterman paths $\Kl_{p^n}(t;(a,b_0))$ once and for all into $p-1$ subfamilies indexed additionally by $a_1\in(\mathbb{Z}/p\mathbb{Z})^{\times}$, and consider the restriction of the map \eqref{kl-pn-def} to the set $(\mathbb{Z}/p^n\mathbb{Z})^{\times}_{a_1}=\big\{a\in(\mathbb{Z}/p^n\mathbb{Z})^{\times}:a\equiv a_1\bmod p\big\}$ and obtain a random variable
\begin{equation}
\label{kl-pn-a1-def}
\Kl_{p^n}(t;(a_1,b_0)):(\mathbb{Z}/p^n\mathbb{Z})^{\times}_{a_1}\to C^0([0,1],\mathbb{C}),\quad a\mapsto\Kl_{p^n}(t;(a,b_0))
\end{equation}
on the finite probability space $(\mathbb{Z}/p^n\mathbb{Z})^{\times}_{a_1}$.

Consider the absolutely continuous probability Borel measure $\mu$ on $[-2,2]$ given by
\begin{equation}
\label{def-mu}
\mu(f)=\frac{1}{\pi}\int_{-2}^2 \frac{f(x)}{\sqrt{4-x^2}}\,\dd x
\end{equation}
for any real continuous function $f$ on $[-2,2]$. From this we construct a finite collection of $C^0([0,1],\mathbb{C})$-valued random variables $\Kl(t;p;(a_1,b_0))$, whose definition and properties we collect in the following proposition.

\begin{proposition}
\label{properties-prop}
Let $p$ be a fixed odd prime, and let $a_1,b_0\in(\mathbb{Z}/p\mathbb{Z})^{\times}$.

Let $(U^{\sharp}_h)_{h\in\mathbb Z,(a_1-h)b_0 \in (\mathbb Z/p\mathbb Z)^{\times 2}}$ be a sequence of independent identically distributed random variables of probability law $\mu$ in \eqref{def-mu}. Then, the random series
\begin{equation}
\label{our-random-series}
\Kl(t;p;(a_1,b_0))=\sum_{\substack{h\in\mathbb Z\\ (a_1-h)b_0 \in (\mathbb Z/p\mathbb Z)^{\times 2}}}\frac{e(ht)-1}{2\pi ih}U^{\sharp}_h \qquad(t\in [0,1]),
\end{equation}
converges almost surely and in law, where the term $h=0$, if present, is interpreted as $tU^{\sharp}_0$. Its limit, as a random function, is almost surely continuous and nowhere differentiable. Moreover, for every $t\in [0,1]$,
$\mathbb{E}(\Kl(t;p;(a_1,b_0)))=0$ and $\mathbb{V}(\Kl(t;p;(a_1,b_0)))\leqslant t$.
\end{proposition}

Our main result, then, is as follows.

\begin{theorem}
\label{main}
Let $p$ be a fixed odd prime, and let $a_1,b_0\in(\mathbb{Z}/p\mathbb{Z})^{\times}$. Then, the sequence of $C^0([0,1],\mathbb C)$-valued random variables $\Kl_{p^n}(t;(a_1,b_0))$ on $(\mathbb Z/p^n\mathbb Z)^{\times}_{a_1}$ defined in \eqref{kl-pn-a1-def} converges in law, as $n\to\infty$, to the $C^0([0,1],\mathbb C)$- valued random variable $\Kl(t;p;(a_1,b_0))$ defined in \eqref{our-random-series}.
\end{theorem}

We prove Theorem~\ref{main} in two steps, by establishing convergence of $\Kl_{p^n}(\cdot;(a_1,b_0))\to\Kl(\cdot;p;(a_1,b_0))$ as $n\to\infty$ in the sense of finite distributions in Sections~\ref{core-sec} and \ref{sums-of-products-sec}, and then proving that the sequence of random variables $\Kl_{p^n}(\cdot;(a_1,b_0))$ is tight as $n\to\infty$ in Section~\ref{in-law-sec}. Combining these two conclusions yields the proof of Theorem~\ref{main} in \S\ref{proof-of-main}.

\begin{remark}
Convergence in the sense of finite distributions in Theorem~\ref{main} is the analogue of the results of Kowalski--Sawin~\cite{KowalskiSawin2016} and Ricotta--Royer~\cite{RicottaRoyer2018}, which for the random variable $\Kl_{p^n}$ with $n=1$ and $n\geqslant 2$ fixed, respectively, and $p\to\infty$ find limiting random Fourier series
\begin{equation}
\label{KKl}
\begin{alignedat}{9}
&\mathrm{K}(t)&&=\sum_{h\in\mathbb{Z}}\frac{e(ht)-1}{2\pi ih}\mathrm{ST}_h, &\qquad& \dd\mu_{\mathrm{ST}}&&=\frac1{\pi}\sqrt{1-(x/2)^2}\,\dd x,\\
&\Kl(t)&&=\sum_{h\in\mathbb{Z}}\frac{e(ht)-1}{2\pi ih}U_h, &&\mu_U&&=\frac12\delta_0+\frac12\mu,
\end{alignedat}
\end{equation}
where $\mathrm{ST}_h$ and $U_h$ are independent identically distributed random variables of probability law $\mu_{\mathrm{ST}}$ and $\mu_U$, respectively.

\begin{figure}
\centering
\includegraphics[width=0.5\textwidth]{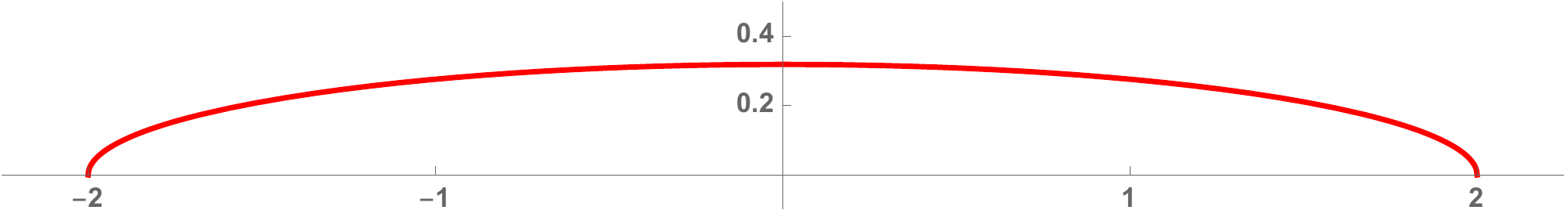}
\medskip
\includegraphics[width=0.5\textwidth]{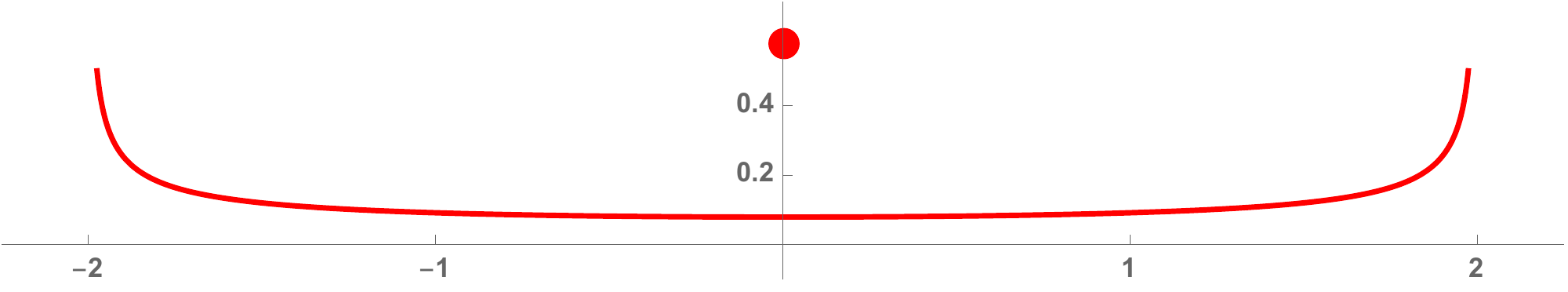}
\medskip
\includegraphics[width=0.5\textwidth]{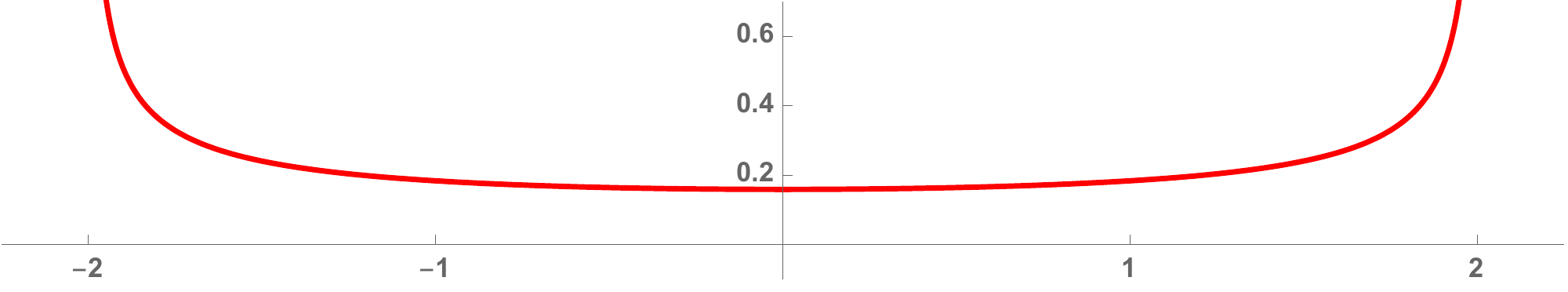}
\caption{Sato--Tate distributions $\mu_{\mathrm{ST}}$, $\mu_U$, and $\mu$, which guide the coefficients of limiting random Fourier series \eqref{our-random-series} and \eqref{KKl} for Kloosterman paths modulo $p$ as $p\to\infty$, modulo $p^n$ for $n$ fixed and $p\to\infty$, and modulo $p^n$ for $p$ fixed and $n\to\infty$; drawn to scale. The big dot in the plot of $\mu_U$ indicates $\tfrac12$ of the Dirac measure at 0.}
\end{figure}

That the analytic shape of these three series is similar stems from the use of the completion method (see Lemma~\ref{completion-lemma}), which in each case can be used (along with some rather nontrivial estimates) to show  that in a suitable sense the incomplete sum $\Kl_{j;p^n}(a,b_0)$ may be approximated by
\begin{equation}
\label{after-completion}
\sum_{|h|<p^n/2}\frac{e(hj)-1}{2\pi ihj}\Kl_{p^n}(a-h,b_0).
\end{equation}

The lens provided by \eqref{after-completion} explains the probabilistic distinction among the limiting random series in \eqref{KKl} and \eqref{our-random-series}. The underlying measure $\mu_{\mathrm{ST}}$ in \eqref{KKl} reflects the classical semi-circle Sato--Tate distribution of the normalized Kloosterman sums $\Kl_p(a-h,b_0)$, while the shape of $\mu_U$ reflects a different Sato--Tate measure for Kloosterman sums to prime power moduli $p^n$ with $n\geqslant 2$, which are given by an explicit exponential (with a $p$-adically analytic phase, see see Lemma~\ref{kloost-eval}) and distributed according to $\mu$ when $(a-h)b_0\in(\mathbb{Z}/p\mathbb{Z})^{\times 2}$, and vanish otherwise, a distinction which in the limit $p\to\infty$ (as $a$ ranges through $(\mathbb{Z}/p\mathbb{Z})^{\times}$) may be thought of as a random event of probability $\frac12$ in \eqref{after-completion}. The measures $\mu_{\mathrm{ST}}$ and $\mu$ are the direct images under the trace map of the probability Haar measure on $\mathrm{SU}_2(\mathbb{C})$ and on the normalizer of the maximal torus in $\mathrm{SU}_2(\mathbb{C})$, respectively ~\cite[Remark 1.4]{RicottaRoyer2018}.

In the depth aspect, however, when $p$ is fixed, for a given $h\in\mathbb{Z}$ there is nothing random about the event $(a-h)b_0\in(\mathbb{Z}/p\mathbb{Z})^{\times 2}$: it predictably happens or not depending only on the class of $(a\bmod p,b_0)\in((\mathbb{Z}/p\mathbb{Z})^{\times})^2$. This in turn induces finitely many (precisely up to $2(p-1)$) distinct limiting distributions in Theorem~\ref{main}, with the surviving terms guided directly by the distribution $\mu$.
\end{remark}

\begin{remark}
\label{all-a-remark}
Fix an integer $b_0\in\mathbb{Z}$ with $p\nmid b_0$. One way to informally think of Theorem~\ref{main} is that the ensemble of all paths $\Kl_{p^n}(t;(a,b_0))$ ($a\in(\mathbb{Z}/p^n\mathbb{Z})^{\times}$), which for a fixed $n$ and $p\to\infty$ has the limiting distribution $\Kl$ shown in \eqref{KKl} (and which is the same for all $n\geqslant 2$), for a fixed $p$ splits into $p-1$ classes according to $a\bmod p$, which simply have different limiting distributions $\Kl(t;p;(a_1,b_0))$ as $n\to\infty$.

One might wonder if a common distribution is somehow restored if one subsequently takes $p\to\infty$. Indeed, let $(\Omega,\mathcal{F},\nu)$ be the probability space such that $\Omega=\bigsqcup_{a_1\in(\mathbb{Z}/p\mathbb{Z})^{\times}}\Omega_{a_1}$, $\mathcal{F}=\mathcal{S}\big(\bigsqcup_{a_1\in(\mathbb{Z}/p\mathbb{Z})^{\times}}\mathcal{F}_{a_1}\big)$, $\nu(\Omega_{a_1})=1/(p-1)$ for every $a_1\in(\mathbb{Z}/p\mathbb{Z})^{\times}$, and $(\Omega_{a_1},\mathcal{F}_{a_1},(p-1)\nu\vert_{\mathcal{F}_{a_1}})$ is an underlying probability space for the $C^0([0,1],\mathbb{C})$-valued random variable $\Kl(\cdot;p;(a_1,b_0))$ in \eqref{our-random-series}. Further, let $\Kl^{\bullet}(\cdot;p)$ be the $C^0([0,1],\mathbb{C})$-valued random variable on $(\Omega,\mathcal{F},\nu)$ defined by
\begin{equation}
\label{Kl-bullet}
\Kl^{\bullet}(\cdot;p)\Big|_{\Omega_{a_1}}=\Kl(\cdot;p;(a_1,b_0))\quad\text{for every }a_1\in(\mathbb{Z}/p\mathbb{Z})^{\times}.
\end{equation}
Then Theorem~\ref{main} implies the following statement, which can be thought of as gluing together the individual limits of $\Kl_{p^n}(\cdot;(a_1,b_0))\to\Kl(\cdot;p;(a_1,b_0))$ in law as $n\to\infty$.

\begin{corollary}
Fix an odd prime $p$ and $b_0\in\mathbb{Z}$ with $p\nmid b_0$. Then the sequence of $C^0([0,1],\mathbb{C})$-valued random variables $\Kl_{p^n}$ given in \eqref{kl-pn-def} converges in law, as $n\to\infty$, to the random variable $\Kl^{\bullet}(\cdot;p)$ defined in \eqref{Kl-bullet}.
\end{corollary}

\begin{figure}
\begin{tikzcd}[row sep=large,column sep=large]
\textcolor{blue}{\Kl_{p^n}(\cdot;(a_1,b_0))} \arrow[d,color=blue,"n\to\infty"] \arrow[r,color=blue,dashed,"\sqcup"] & \Kl_{p^n} \arrow[r,"p\to\infty"] \arrow[d,color=blue,"n\to\infty"] & \Kl \arrow[d,equal,"n\to\infty"] \\
\textcolor{blue}{\Kl(\cdot;p;(a_1,b_0))} \arrow[r,color=blue,dashed,"\sqcup"] & \textcolor{blue}{\Kl^{\bullet}(\cdot;p)} \arrow[r,color=blue,"p\to\infty"] &\Kl
\end{tikzcd}
\caption{Commutative diagram of convergence in law, with new entries corresponding to Theorems~\ref{main} and \ref{large-p-theorem} marked in blue.}
\label{diagram-of-convergence}
\end{figure}

Further, in Theorem~\ref{large-p-theorem}, we prove that the random variable $\Kl^{\bullet}(t;p)$
converges in law to the $C^0([0,1],\mathbb{C})$-valued random variable $\Kl(t)$ given in \eqref{KKl} as $p\to\infty$. This restores harmony between the results of Theorem~\ref{main} and the results of \cite{RicottaRoyer2018,RicottaRoyerShparlinski2020}, as shown in Figure~\ref{diagram-of-convergence}.
\end{remark}

\begin{remark}
Kloosterman sums to prime power moduli exhibit another, elementary but perhaps less well popularized, curious property: if $ab\in(\mathbb{Z}/p^n\mathbb{Z})^{\times 2}$ (as is the case whenever the complete sum $\Kl_{p^n}(a,b)\neq 0$), some of the Kloosterman summands $e_{p^n}(ax+b\bar{x})$ appear with very high multiplicity. This is perhaps surprising in the light of the classical evaluation of Kloosterman sums (Lemma~\ref{kloost-eval}), whose two summands arise from provably \emph{non-singular} stationary points. It also makes for somewhat startling numerics for small $p^n$; for example, modulo 27 there are only four distinct summands (two pairs of complex conjugates).

The phenomenon is not difficult to understand. Taking for concreteness $a=b=1$ (other cases being just a change of variable away), for $1\leqslant\kappa<n/2$ and $u\in\pm 2+p^{2\kappa}(\mathbb{Z}/p^n\mathbb{Z})^{\times}$, the congruence $x+\bar{x}\equiv u\pmod{p^n}$ has exactly $2p^{\kappa}$ solutions; in particular the multiplicity becomes as high as $\asymp p^{n/2}$. In particular, for $(ab/p)=1$, as $n\to\infty$, an asymptotically \emph{positive proportion} ($100\%$ for $p=3$) of Kloosterman fractions appear with multiplicity higher than 2. On the balance between a high mutliplicity $2p^{\kappa}$ and the fact that the $\asymp p^{n-2\kappa}$ distinct such terms still follow the generic distribution, however, there is no impact on the limiting distribution in Theorem~\ref{main}. The phenomenon is not visible in the stationary phase analysis of the Kloosterman sum because the latter is concerned with roots of the derivative of the phase rather than of the phase itself. 

Nevertheless, possibly highly singular stationary points do play a crucial role in the stationary phase analysis described in \S\ref{sums-of-products-intro}, including shifts by frequencies $h$ in \eqref{after-completion} that are high powers $p^{2\kappa}$ and would detect any impact of highly popular fractions $(x+\bar{x})/p^n=u/p^n$ with $u\in\pm 2+p^{2\kappa}(\mathbb{Z}/p^n\mathbb{Z})^{\times}$.
\end{remark}

\subsection{Sums of products}
\label{sums-of-products-intro}
At the arithmetic heart of the proof of Theorem~\ref{main} are estimates on sums of products of Kloosterman sums of the form
\begin{equation}
\label{sum-of-products-intro-eq}
\sum_{a\in(\mathbb{Z}/p^n\mathbb{Z})^{[T]}}\prod_{\tau\in T}\Kl_{p^n}(a-\tau,b_0)^{\mu(\tau)},
\end{equation}
for $T\subseteq\mathbb{Z}/p^n\mathbb{Z}$, arbitrary integers $(\mu(\tau))_{\tau\in T}$, and summation being over residues $a\in(\mathbb{Z}/p^n\mathbb{Z})^{\times}$ such that $(a-\tau)b_0\in(\mathbb{Z}/p\mathbb{Z})^{\times 2}$ for every $\tau\in T$. In the case of prime moduli, such sums have been studied by Fouvry--Kowalski--Michel using deep tools of algebraic geometry, with spectacular applications~\cite{FouvryKowalskiMichel2015}. In the prime power case, Kloosterman sums can be explicitly evaluated, and, in \cite{RicottaRoyer2018}, where the modulus is a fixed power $p^n$, the phase in the exponential sum resulting from \eqref{sum-of-products-intro-eq} is replaced with a polynomial of fixed degree, and sums are estimated using a Weyl bound (essentially repeated differencing). The resulting bound \cite[Proposition 4.10]{RicottaRoyer2018}, however, degenerates badly with increasing $n$, so this route is not available as $n\to\infty$.

Instead, in the properly depth aspect, the $p$-adic method of stationary phase (see Lemma~\ref{statphase-lemma}) can be applied and leads to a condition roughly of the form
\begin{equation}
\label{statphase-cond}
\sum_{\tau\in T}\epsilon_{\tau}((a-\tau)b_0)^{-1}_{1/2}\equiv 0\pmod{p^{\lfloor n/2\rfloor}},
\end{equation}
where $(\cdot)_{1/2}$ is a branch of the $p$-adic square root (see \S\ref{sqroot-section}). In a pleasant application of the method of stationary phase, one hopes to argue that the stationary points in \eqref{statphase-cond} are non-singular, or, barring that, at least not overly singular, so that a version of Hensel's lemma applies. The possibility of singular stationary points, of which there may be very many, is a known obstacle in the estimates of exponential sums to prime power moduli, such as for example in the long-standing restriction of Burgess' bound on character sums to cube-free moduli. With fewer summands the number of singular stationary points can sometimes be controlled (see, for example, \cite[Lemma 7]{Heath-Brown1978a}), but in our case the method of moments requires that we allow an arbitrary number of summands in \eqref{statphase-cond}, an algebraically and combinatorially forbidding situation, and we must contend with the possibility of very many highly singular solutions. In fact, all solutions to \eqref{statphase-cond} could be singular if there are sufficiently high collusions among the $\tau\in T$! In Theorems~\ref{kloosterman-moments} and ~\ref{sums-of-products-theorem}, we show that such high collusions are in fact the only possibility for failure of power cancellation in \eqref{sum-of-products-intro-eq}, which clears the way to Theorem~\ref{main}. These theorems are of independent interest, and in fact Section~\ref{sums-of-products-sec} introduces a method that applies to many more exponential sums to high prime power moduli; see Remark~\ref{more-general-remark}.

\begin{remark}
Discussion in \S\ref{sums-of-products-intro} already shows that the depth-aspect limit $n\to\infty$ behaves entirely differently from the large $p$ limit in \cite{RicottaRoyer2018}. In the depth aspect, the phase in \eqref{sum-of-products-intro-eq} cannot be productively thought of as a polynomial of fixed degree, but rather more like a $p$-adically analytic function. We estimate sums of products using a stationary phase analysis including potentially highly singular critical points, which is also reflected in the shape of Theorem~\ref{kloosterman-moments}. Here we mention several other key differences.

Any analysis of a sum like \eqref{sum-of-products-intro-eq} is bound to run into difficulties as shifts $\tau\in T$  collude to a certain extent, as in $\tau\equiv\tau'\pmod{p^{\kappa}}$ for some $\tau\neq\tau'\in T$, since this is transitionary behavior between the generic and fully aligned cases. In the $p\to\infty$ limit, such difficulties can often be estimated away trivially as in \cite[Lemma 4.14]{RicottaRoyer2018}, since a condition like $\tau\equiv\tau'\pmod p$ typically happens with a ``small'' frequency $1/p$. Such an approach is, of course, not available for a fixed $p$, and more generally being divisible by $p$ or collusions modulo $p$ are not exceptional events unless they occur to an increasing power $p^m$.

Considering all paths $\Kl_{p^n}(t;(a,b_0))$ ($a\in(\mathbb{Z}/p^n\mathbb{Z})^{\times}$) as one ensemble as in Remark~\ref{all-a-remark} and \cite{RicottaRoyer2018}, a critical piece in the evaluation of the main term becomes the quantity
\[ \big|(\mathbb{Z}/p^n\mathbb{Z})^{[T]}\big|=\big|\big\{a\in(\mathbb{Z}/p^n\mathbb{Z})^{\times}:(a-\tau)b_0\in(\mathbb{Z}/p\mathbb{Z})^{\times 2}\text{ for every }\tau\in T\big\}\big|, \]
which of course essentially only depends on $(T+p\mathbb{Z})/p\mathbb{Z}$. In \cite[Proposition 4.8]{RicottaRoyer2018}, this quantity is denoted by $|A_{p^n}(\mu(\tau))|$ and estimated, for $p\to\infty$, using Weil's version of Riemann Hypothesis. For $p$ fixed, however, there is seemingly no rhyme or reason to the values of $|(\mathbb{Z}/p\mathbb{Z})^{[T]}|$, and increasingly so for larger $|T|$. It is in fact at this point that one might realize that the different classes of $a\bmod p$ need to be separated into distinct ensembles.
\end{remark}

\subsection{Rearrangement}
A complete Kloosterman sum is a natural algebro-geometric object, but an incomplete sum entails a choice of ordering of terms. For a prime modulus, there is only one ordering (the obvious one) that could be reasonably construed as natural, but for a more structured modulus there are other perfectly reasonable ways of summing, which then lead to different Kloosterman paths. We illustrate this point with a simple example.

Consider the function $f_{p^n}(\cdot;(a,b)):(\mathbb{Z}/p^{n-1}\mathbb{Z})^{\times}\to\mathbb{C}$ that groups $p$ terms in $\Kl_{p^n}(a,b)$ together:
\[ f_{p^n}(x;(a,b))=\sum_{k\bmod p}e_{p^n}\big(a(x+kp^{n-1})+b\overline{(x+kp^{n-1})}\big). \]
Since the prime $p$ is fixed, the partial sums
\begin{equation}
\label{rearranged-partial}
\Kl^{\circ}_{j;p^n}(a,b)=\frac1{p^{n/2}}\sum_{1\leqslant x\leqslant j,\,p\nmid x}f_{p^n}(x;(a,b))
\end{equation}
correspond to only a slight reordering of the complete Kloosterman sums $\Kl_{p^n}(a,b)=\Kl^{\circ}_{p^{n-1};p^n}(a,b)$. Following the steps in \S\ref{paths-intro-sec} one can correspondingly form modified Kloosterman paths $\gamma^{\circ}_{p^n}(a,b)$ by concatenating the closed segments $[\Kl^{\circ}_{j_i;p^n}(a,b_0),\Kl^{\circ}_{j_{i+1};p^n}(a,b_0)]$, and parametrize the paths $\gamma^{\circ}_{p^n}(a,b)$ by continuous functions $\Kl^{\circ}_{p^n}\in C^0([0,1],\mathbb{C})$.

Since $\overline{x+kp^r}\equiv\overline{x}-\overline{x}^2\cdot kp^r\pmod{p^{\min(2r,n)}}$ for $1\leqslant r\leqslant n$ and $x\in(\mathbb{Z}/p^n\mathbb{Z})^{\times}$, we find that in fact
\[ \Kl^{\circ}_{j;p^n}(a,b)=\frac{p}{p^{n/2}}\sum_{\substack{1\leqslant x\leqslant j, p\nmid x,\\ x^2\equiv \bar a b \bmod{p}}}e_{p^n}(ax+b\overline{x}), \]
and in particular $\Kl_{j;p^n}(a,b)=0$ (making for a boring zero path) unless $a\in b(\mathbb{Z}/p^n\mathbb{Z})^{\times 2}$. Restricting to the latter case, for a fixed $b_0\in\mathbb{Z}$ with $p\nmid b_0$, the map from $b_0(\mathbb Z/p^n\mathbb Z)^{\times 2}\to C^0([0,1],\mathbb C)$ given by
\[ a\mapsto\Kl^{\circ}_{p^n}(t;(a,b_0)) \]
can be viewed as a $C^0([0,1],\mathbb{C})$-valued random variable on the finite probability space $b_0(\mathbb Z/p^n\mathbb Z)^{\times 2}$ with the uniform probability measure, which we denote by $\Kl^{\circ}_{p^n}(t)$.

The following theorem shows that the slight rearrangement of terms that led to \eqref{rearranged-partial} produces Kloosterman paths with a universal limiting distribution. Specifically,  let $(U^{\sharp}_h)_{h\in\mathbb Z}$ be a sequence of independent identically distributed random variables of probability law $\mu$ in \eqref{def-mu}. One shows exactly as in Proposition~\ref{properties-prop} that the random series
\begin{equation}
\label{rearr-random-series}
\Kl^{\circ}(t)=\sum_{h\in\mathbb Z}\frac{e(ht)-1}{2\pi ih}U^{\sharp}_h \qquad(t\in [0,1]),
\end{equation}
where the term $h=0$ is interpreted as $tU^{\sharp}_0$, converges almost surely and in law to a random function with properties as in Proposition~\ref{properties-prop}. Then the following holds.

\begin{theorem}
\label{rearranged-paths-thm}
Let $p$ be a fixed odd prime, and let $b_0\in\mathbb{Z}$, $p\nmid b_0$. 
Then the sequence of $C^0([0,1],\mathbb C)$-valued random variables $\Kl^{\circ}_{p^n}(t)$ on $b_0(\mathbb Z/p^n\mathbb Z)^{\times 2}$ converges in law to the $C^0([0,1],\mathbb C)$-valued random variable $\Kl^{\circ}(t)$ as $n \to\infty$.
\end{theorem}

\begin{figure}
    \centering
    \includegraphics[width=0.5\textwidth]{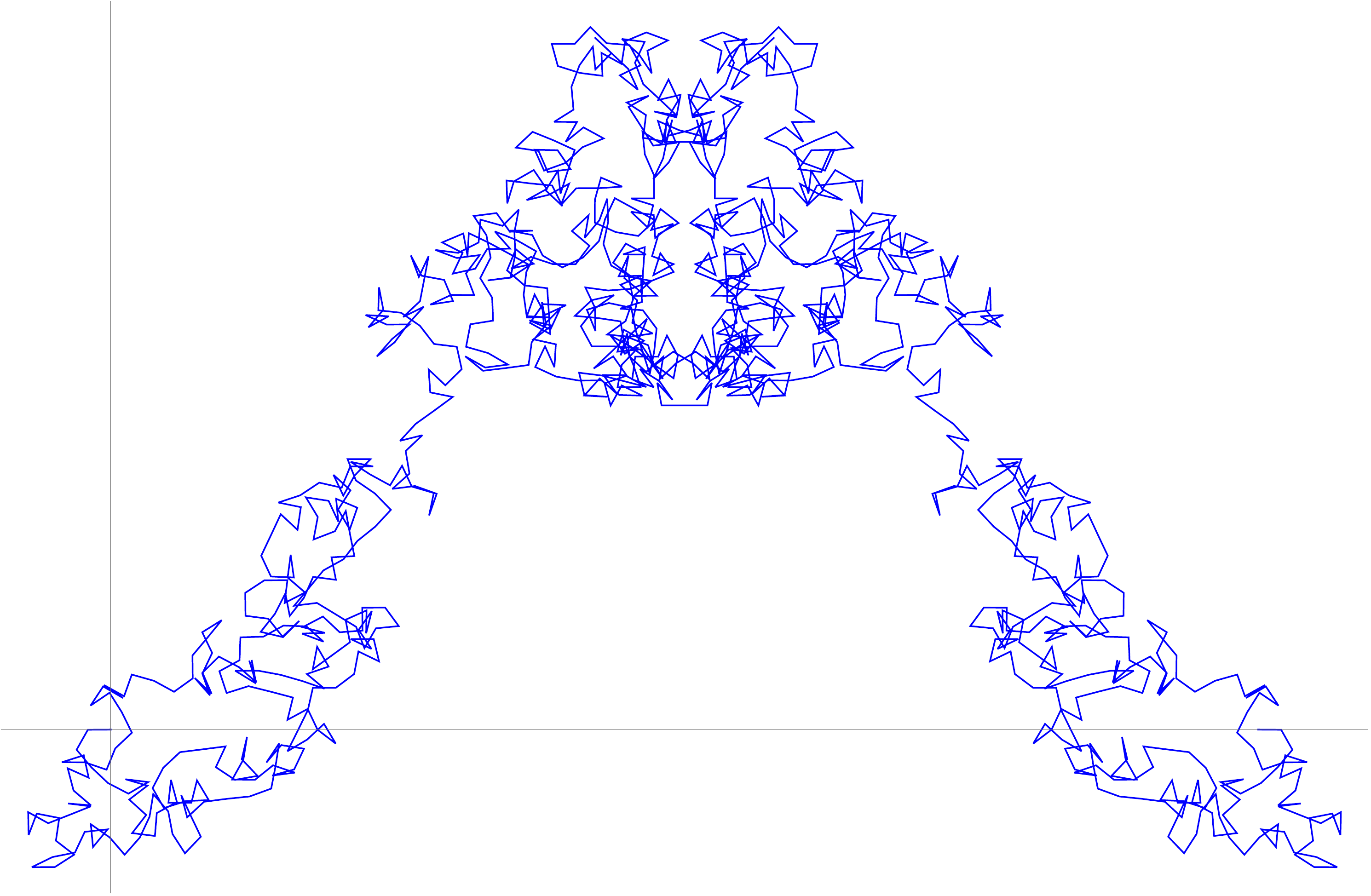}
    \caption{The Plot of the modified Kloosterman path $\gamma^{\circ}_{5^6}(9;1)$}
    \label{fig:my_label}
\end{figure}

\subsection*{Notation:} Throughout the paper, $p$ denotes a fixed odd prime. The notation $\sum^{\times}$ indicated that the summation is restricted to integers (or congruence classes) coprime to $p$. For $n\in\mathbb{Z}$ and $\tau\in\mathbb{Z}_{\geqslant 0}$, we write $p^{\tau}\exmid n$ or $\ord_pn=\tau$ to indicate that $p^{\tau}\mid n$ and $p^{\tau+1}\nmid n$; we apply the same notation to congruence classes modulo $p^t$ with $t\geqslant\tau+1$. For $x\in(\mathbb{Z}/p^n\mathbb{Z})^{\times}$ we write $\bar{x}$ for the inverse of $x$ modulo $p^n$, where the value of $n$ is clear from the context and cannot cause confusion, and we also denote $x^{-n}=\bar{x}^n$ for $n\in\mathbb{N}$.

We write $f=\mathrm{O}(g)$ or $f\ll g$ to denote that $|f|\leqslant Cg$ for some constant $C>0$, which may depend on $p$ but is otherwise absolute except as indicated by a subscript. While this is not important for us, the dependence of all constants on $p$ may easily be explicated and is always at most polynomial. Exceptionally, in \S\ref{large-p-section}, which deals with the limit $p\to\infty$, all implied constants are absolute (independent of $p$), unless otherwise indicated by a subscript.

As customary, we write $e(z)=e^{2\pi iz}$, and, for $x\in\mathbb{Z}$, we sometimes write $e_{p^n}(x)=e(x/p^n)=e^{2\pi x/p^n}$. We denote the cardinality of a finite set $S$ by $|S|$, and we use $A\sqcup B$ to denote shorthand the disjoint union of sets. Finally, $\delta_x$ denotes the Dirac measure at $x$, namely $\delta_x(E)=1$ if $x\in E$ and $0$ otherwise, where the underlying set and $\sigma$-algebra are clear from the context.

\section{Preliminaries}

\subsection{Probabilistic notions and tools}
\label{prob-sec}
In this section, we collect some standard facts about probability in complex Banach spaces, and in particular in the space $C^0([0,1],\mathbb{C})$ endowed with the sup-norm. We lean on \cite[Appendix A]{RicottaRoyer2018}, and the proofs can be found in \cite[\S{}B.11]{KowalskiBook}.

\begin{definition}
Let $Y$ be a Banach space, let $(X_n)$ be a sequence of $Y$-valued random variables on probability spaces $(\Omega_n,\mathcal{F}_n,\nu_n)$, and let $X$ be a $Y$-valued random variable on $(\Omega,\mathcal{F},\nu)$.
\begin{enumerate}
\item If $(\Omega_n,\mathcal{F}_n,\nu_n)=(\Omega,\mathcal{F},\nu)$ for every $n\in\mathbb{N}$, we say that $(X_n)$ converges \emph{almost surely} if $\nu(\{\omega\in\Omega:\lim X_n(\omega)\text{ exists in }Y\})=1$, and we say that $X_n$ converges to $X$ almost surely if $\nu(\{\omega\in\Omega:\lim X_n(\omega)=X(\omega)\})=1$.
\item We say that $(X_n)$ converges \emph{in law} to $X$ if, for every continuous and bounded map $\varphi:Y\to\mathbb{C}$, $\lim\mathbb{E}(\varphi(X_n))=\mathbb{E}(\varphi(X))$.
\item If $Y=C^0([0,1],\mathbb{C})$, we say that $(X_n)$ converges to $X$ in the sense of \emph{finite distributions} if, for every $k\geqslant 1$ and all $0\leqslant t_1<t_2<\dots<t_k\leqslant 1$, the sequence of $\mathbb{C}^k$-valued random vectors $(X_n(t_1),\dots,X_n(t_k))$ converges in law to the $\mathbb{C}^k$-valued random vector $(X(t_1),\dots,X(t_k))$.
\item If $Y$ is separable, we say that the sequence $(X_n)$ is \emph{tight} if, for every $\epsilon>0$, there exists a compact subset $K\subseteq Y$ such that, for every $n\geqslant 1$, $\nu_n(\{X_n\in K\})\geqslant 1-\epsilon$.
\end{enumerate}
\end{definition}

The notion of tightness of a sequence of random variables provides a practical bridge from convergence in the sense of finite distributions to the substantially stronger claim of convergence in law, using the following criteria.

\begin{proposition}[Prokhorov's criterion]
\label{prokhorov}
Let $(X_n)_{n=1}^{\infty}$ and $X$ be $C^0([0,1],\mathbb{C})$-valued random variables. If the sequence $(X_n)$ converges to $X$ as $n\to\infty$ in the sense of finite distributions, and if the sequence $(X_n)$ is tight, then the sequence $(X_n)$ converges to $X$ as $n\to\infty$ in law.
\end{proposition}

\begin{proposition}[Kolmogorov's criterion for tightness]
\label{kolmogorov}
Let $(X_n)_{n=1}^{\infty}$ be a sequence of $C^0([0,1],\mathbb{C})$-valued random variables. If there exist $\alpha,\delta>0$ such that
\[ \forall s,t\in [0,1]^2,\quad \mathbb{E}\big(|X_n(s)-X_n(t)|^{\alpha}\big)\ll|s-t|^{1+\delta}, \]
then the sequence $(X_n)$ is tight.
\end{proposition}

\subsection{Square roots modulo \texorpdfstring{$p^n$}{p to n}}
\label{sqroot-section}

The explicit evaluation of Kloosterman sums $\Kl_{p^n}(a,b)$ to a proper prime power modulus (see Lemma~\ref{kloost-eval}) features exponentials with phases that are solutions of congruences of the form $x^2\equiv ab\pmod{p^r}$, that is, square roots modulo high powers of $p$. In this section, we describe the construction and properties of these square roots. We keep in mind the paradigm that these are really restrictions of branches of the $p$-adic square root (in analogy with the exponentials appearing in the asymptotics of $J$-Bessel functions, the archimedean analogue of Kloosterman sums) but avoid the $p$-adic language. For more details and a fully $p$-adic perspective, we refer to \cite[\S2.4]{BlomerMilicevic2015}.

For every $r\in(\mathbb{Z}/p\mathbb{Z})^{\times 2}$, there are exactly two classes $s\in\mathbb{Z}/p\mathbb{Z}$ such that $s^2=r$. We fix once and for all a choice function $s:(\mathbb{Z}/p\mathbb{Z})^{\times 2}\to(\mathbb{Z}/p\mathbb{Z})^{\times}$ such that $s(r)^2=r$ for every $r\in(\mathbb{Z}/p\mathbb{Z})^{\times 2}$ (any of the $2^{(p-1)/2}$ such choices will do). For every $x\in(\mathbb{Z}/p^n\mathbb{Z})^{\times 2}$, by Hensel's lemma (Lemma~\ref{hensel-lemma} in the non-singular case $\rho=0$) there exists as unique $u\in(\mathbb{Z}/p^n\mathbb{Z})^{\times}$ such that $u^2\equiv x\pmod{p^n}$ and $u\equiv s(x)\pmod p$; this gives rise to a square root $u:(\mathbb{Z}/p^n\mathbb{Z})^{\times 2}\to(\mathbb{Z}/p^n\mathbb{Z})^{\times}$, which we also denote by $u=u_{1/2}(x)=x_{1/2}$ and write $x_{1/2}^k=(x_{1/2})^k$ for $k\in\mathbb{Z}$.

The square-root on $(\mathbb{Z}/p^n\mathbb{Z})^{\times\, 2}$ thus defined satisfies the following differentiability-like property for $x\in(\mathbb{Z}/p^n\mathbb{Z})^{\times\, 2}$, $t\in\mathbb{Z}/p^n\mathbb{Z}$, $\kappa\geqslant 1$, and $m\in\mathbb{Z}$:
\begin{equation}
\label{sqroot-diff}
(x+p^{\kappa}t)^m_{1/2}\equiv x^m_{1/2}+\tfrac12 m\cdot x^{m-2}_{1/2}\cdot p^{\kappa}t+\tfrac18 m(m-2)\cdot x^{m-4}_{1/2}\cdot p^{2\kappa}t^2\pmod{p^{\min(3\kappa,n)}}.
\end{equation}
This is easily verified, since the squares of both sides agree modulo $p^{2\kappa}$, and both sides agree modulo $p$; cf.~\cite[(2.6)]{BlomerMilicevic2015}.

\subsection{Method of stationary phase}
In this section, we collect facts about the so-called $p$-adic method of stationary phase, a powerful tool in the study of complete exponential sums modulo prime powers analogous to the classical method of stationary phase for oscillatory exponential integrals. We phrase our results here with emphasis on differentiability-like properties \eqref{diffble-eq1} and \eqref{diffble-eq2} but without invoking $p$-adic language. For more details, we refer to \cite[Lemmata 12.2 and 12.3]{IwaniecKowalski2004} for a formulation with phases that are rational functions, or to \cite[Lemma~7]{BlomerMilicevic2015} for a general statement with a more $p$-adic analytic perspective.

\begin{lemma}[Method of stationary phase]
\label{statphase-lemma}
Let $1\leqslant\kappa_0\leqslant n$, and let $T\subseteq\mathbb{Z}/p^n\mathbb{Z}$ be a set invariant under translations by $p^{\kappa_0}\mathbb{Z}/p^n\mathbb{Z}$.
\begin{enumerate}
\item\label{MSP-claim1} Suppose that functions $f,f_1:T\to\mathbb{Z}/p^n\mathbb{Z}$ satisfy
\begin{equation}
\label{diffble-eq1}
f(x+p^{\kappa}t)\equiv f(x)+f_1(x)\cdot p^{\kappa}t\pmod{p^{2\kappa}}
\end{equation}
for all $x\in T$, $t\in\mathbb{Z}/p^n\mathbb{Z}$, and $\kappa\geqslant\kappa_0$. Then, for every $\kappa\geqslant\max(\kappa_0,n/2)$, the set $\{x\in T:f_1(x)\equiv 0\bmod p^{n-\kappa}\}$ is invariant under translations by $p^{\kappa}\mathbb{Z}/p^n\mathbb{Z}$, and
\[ \sum_{x\in T}e\left(\frac{f(x)}{p^n}\right)=\sum_{\substack{x\in T\\f_1(x)\equiv 0\bmod{p^{n-\kappa}}}}e\left(\frac{f(x)}{p^n}\right)=p^{n-\kappa}\sum_{\substack{x\in T/p^{\kappa}\mathbb{Z}\\f_1(x)\equiv 0\bmod{p^{n-\kappa}}}}e\left(\frac{f(x)}{p^n}\right). \]
\item Suppose that functions $f,f_1:T\to\mathbb{Z}/p^n\mathbb{Z}$, $f_2:T\to(\mathbb{Z}/p^n\mathbb{Z})^{\times}$  satisfy
\begin{equation}
\label{diffble-eq2}
f(x+p^{\kappa}t)\equiv f(x)+f_1(x)\cdot p^{\kappa}t+\bar{2}f_2(x)\cdot p^{2\kappa}t^2\pmod{p^{2\kappa+1}}
\end{equation}
for all $x\in T$, $t\in\mathbb{Z}/p^n\mathbb{Z}$, and $\kappa\geqslant\kappa_0$. Then,
writing $n=2\kappa+\rho$ with $\rho\in\{0,1\}$,
\[ \sum_{x\in T}e\left(\frac{f(x)}{p^n}\right)=p^{n/2}
\sum_{\substack{x_0\in T/p^{\kappa}\mathbb{Z}\\ f_1(x_0)\equiv 0\bmod p^{\kappa}}}
\epsilon\big(2f_2(x_0),p^{\rho}\big)e\bigg(\frac{f(x_0)-\overline{2f_2(x_0)}{f_1(x_0)^2}}{p^n}\bigg), \]
where
$\epsilon(\cdot,1)=1$ and $\epsilon(\cdot,p)=(\cdot/p)i^{(\iota-1)/2}$ for $p\equiv \iota\bmod 4$, $\iota\in\{1,3\}$.
\end{enumerate}
\end{lemma}

\begin{proof}
For $p\nmid t$, the invariance of the set $\{x\in T:f_1(x)\equiv 0\bmod p^{n-\kappa}\}$ under translation by $p^{\kappa}t$ follows immediately from applying \eqref{diffble-eq1} in the form $f(x)\equiv f(x+p^{\kappa}t)+f_1(x+p^{\kappa}t)\cdot p^{\kappa}(-t)\pmod{p^{2\kappa}}$; otherwise, we bootstrap from the same conclusion for $t'=1$ and $t'=t-1$. The remaining claim in \eqref{MSP-claim1} follows immediately by orthogonality from
\[ \sum_{x\in T}e\left(\frac{f(x)}{p^n}\right)=\frac1{p^{n-\kappa}}\sum_{x\in T}e\left(\frac{f(x)}{p^n}\right)\sum_{t\bmod p^{n-\kappa}}e\left(\frac{f_1(x)\cdot p^{\kappa}t}{p^n}\right). \]
The second claim is immediate if $n=2\kappa$. The odd case is similar, using the classical evaluation of the quadratic Gauss sum; see  \cite[Lemma~7]{BlomerMilicevic2015}.
\end{proof}

\begin{lemma}
\label{kloost-eval}
Let $a\in\mathbb{Z}/p^n\mathbb{Z}$, $b\in(\mathbb{Z}/p^n\mathbb{Z})^{\times}$. Then,  $\Kl_{p^n}(a,b)=0$ if $ab\not\in (\mathbb Z/p^n\mathbb Z)^{\times 2}$. Otherwise, if $ab\in (\mathbb Z/p^n\mathbb Z)^{\times 2}$,
\[ \Kl_{p^n}(a,b)=2\left(\frac{(ab)_{1/2}}{p^n}\right)\mathrm{Re}\big[\epsilon_{p^n}e_{p^n}\big(2(ab)_{1/2}\big)\big], \]
where $(\cdot)_{1/2}$ refers to the square root introduced in \S\ref{sqroot-section}, $(\cdot/p^n)$ is the Jacobi symbol, and
\[\epsilon_{p^n}=\begin{cases}
1,&\text{if $2\mid n$ or $p\equiv 1\bmod 4$},\\
i, &\text{if $2\nmid n$ and $p\equiv 3\bmod 4$}.
\end{cases}\]
\end{lemma}

Lemma~\ref{kloost-eval} follows immediately by an application of the stationary phase Lemma~\ref{statphase-lemma}, using the differentiability property~\eqref{sqroot-diff}, and finally using the classical evaluation of the quadratic Gauss sum modulo $p$; see, e.g.~\cite[Lemma~8]{BlomerMilicevic2015}, \cite[Eq. (12.39)]{IwaniecKowalski2004}.

\section{Reduction steps and the core argument}
\label{core-sec}
In this section, we collect all required properties of the limit random series $\Kl(t;p;(a_1,b_0))$ defined in \eqref{our-random-series} and perform all steps needed to reduce the claim of convergence in the sense of finite distributions of Theorem~\ref{main} to a crucial estimate on sums of products of Kloosterman sums. This latter estimate is proved in section \ref{sums-of-products-sec}.

Recall that $a_1,b_0\in(\mathbb{Z}/p\mathbb{Z})^{\times}$ are fixed. Let $k\geqslant 1$ be a fixed integer, and let $n\geqslant 2$. Let $\bm{t}=(t_1,...,t_k)$ be a fixed $k$-tuple in $[0,1]^k$,
$\bm{n}=(n_1,\dots,n_k),\bm{m}=(m_1,\dots,m_k)$ be two fixed $k$-tuples of non-negative integers. Let $\ell(\bm{m}+\bm{n})=\sum_{i=1}^k(m_i+n_i)$. We define the complex moments by
\begin{equation}
\label{complex-moment-def}
\mathcal{M}_{p^n}(\bm{t};\bm{m},\bm{n};b_0)=\frac{1}{p^{n-1}}\sum_{\substack{a\in (\mathbb Z/p^n\mathbb Z)^{\times}\\ a\equiv a_1 \bmod p}}\prod_{i=1}^k \overline{\Kl_{p^n}(t_i;(a,b_0))}^{m_i}\Kl_{p^n}(t_i;(a,b_0))^{n_i}.
\end{equation}
Our aim is to prove the following proposition:
\begin{proposition}
\label{complex-moments-prop}
There exists a constant $\delta=\delta(\|\bm{m}\|_1+\|\bm{n}\|_1)>0$ such that the complex moment $\mathcal{M}_{p^n}(\bm{t};\bm{m};\bm{n};b_0)$ defined in \eqref{complex-moment-def} satisfies
\[ \mathcal{M}_{p^n}(\bm{t};\bm{m},\bm{n};b_0)=\mathbb E\Big(\prod_{i=1}^k\overline{\Kl(t_i;p;(a_1,b_0))}^{m_i}\Kl(t_i;p;(a_1,b_0))^{n_i}\Big)+\mathrm{O}(p^{-\delta n}), \]
where $\Kl(t_i;p;(a_1,b_0))$ is the random variable defined by \eqref{our-random-series}.
\end{proposition}

Proposition~\ref{complex-moments-prop} follows immediately from Lemma~\ref{moments-approx-lemma} and Propositions~\ref{main-error-terms-proposition} and \ref{error-term-estimate-proposition} below. Its immediate consequence is the following statement.

\begin{corollary}
\label{finite-distributions-corollary}
The sequence of $C^0([0,1],\mathbb C)$-valued random variables $\Kl_{p^n}(t;(a_1,b_0))$ on $(\mathbb Z/p^n\mathbb Z)^{\times}_{a_1}$ defined in \eqref{kl-pn-a1-def} converges in the sense of finite distributions to the $C^0([0,1],\mathbb C)$- valued random variable $\Kl(t;p;(a_1,b_0))$ as $n\to\infty$.
\end{corollary}

\subsection{The limit random variable}
Recall the absolutely continuous Borel probability distribution $\mu$ given in \eqref{def-mu}. Its moments are given for $m\in\mathbb{Z}_{\geqslant 0}$ by
\begin{equation}
\label{moments-of-mu}
\int_{\mathbb{R}}x^m\,\dd\mu(x)=\frac1{\pi}\int_{-2}^2\frac{x^m}{\sqrt{4-x^2}}\,\dd x=\delta_{2\mid m}\binom{m}{m/2},
\end{equation}
as is easily seen by setting $x=2\cos\theta$. Here, $\delta_{2\mid m}$ is the Kronecker delta and we formally set $\binom 00=1$. Thus for any finite sequence $(U^{\sharp}_i)_{i=1}^r$ of independent random variables identically distributed with probability law $\mu$, and every sequence $(m_i)_{i=1}^r$ of non-negative integers,
\begin{equation}
\label{expectation-of-products}
\mathbb{E}\bigg(\prod_{i=1}^rU^{\sharp m_i}_i\bigg)=\prod_{i=1}^r\delta_{2\mid m_i}\binom{m_i}{m_i/2}.
\end{equation}

For fixed $a_1,b_0\in(\mathbb{Z}/p\mathbb{Z})^{\times}$, let $(U^{\sharp}_h)_{h\in\mathbb Z,(h-a_1)b_0 \in (\mathbb Z/p\mathbb Z)^{\times 2}}$ be a sequence of independent identically distributed random variables of probability law $\mu$. For $H>0$, consider the new random variable
\begin{equation}
\label{KlpH-truncated}
\Kl_H(t;p;(a_1,b_0))=\sum_{\substack{|h|\leqslant H\\ (a_1-h)b_0 \in (\mathbb Z/p\mathbb Z)^{\times 2}}}\frac{e(ht)-1}{2\pi ih}U^{\sharp}_h \qquad(t\in [0,1]).
\end{equation}
When $\Kl_{H}(t;p;(a_1,b_0))$ converges as $H\!\to\infty$, we write the limit as the series $\Kl(t;p;(a_1,b_0))$ shown in \eqref{our-random-series}. The next proposition, which is an elaboration of Proposition~\ref{properties-prop}, shows that this indeed defines a random series a.e.

\begin{proposition}
\label{random-series-properties}
Let $p$ be an odd prime, and let $a_1,b_0\in(\mathbb{Z}/p\mathbb{Z})^{\times}$.
\begin{enumerate}
\item For every $t\in [0,1]$, the random series $\Kl(t;p;(a_1,b_0))$ converges almost surely, hence in law.
\item\label{rate-of-convergence-item}
For every $t\in [0,1]$, $H\geqslant 2$, and $k\geqslant 0$, we have uniformly in $p$
\begin{align*}
\| \Kl_{H}(t;p;(a_1,b_0))\|_\infty&\ll \log H,\\
\mathbb{E}\big(|\Kl(t;p;(a_1,b_0))-\Kl_{H}(t;p;(a_1,b_0))|^k\big)&\ll_k H^{-k/2}.
\end{align*}
\item For every $t\in [0,1]$, the Laplace transform $\mathbb{E}(e^{\lambda\mathrm{Re}\,\Kl(t;p;(a_1,b_0))+\nu\mathrm{Im}\,\Kl(t;p;(a_1,b_0))})$ is well-defined for all $\lambda,\nu\in\mathbb{Z}_{\geqslant 0}$. In particular, $\Kl(t;p;(a_1,b_0))$ has moments of all orders, and
\[ \mathbb E(\Kl(t;p;(a_1,b_0))=0, \quad \mathbb{V}(\Kl(t;p;(a_1,b_0)))\leqslant t. \]
\item The random series $\Kl(t;p;(a_1,b_0))$ is almost surely a continuous, nowhere differentiable function on $[0,1]$.
\end{enumerate}
\end{proposition}

Proposition~\ref{random-series-properties} is proved in the same way as \cite[Proposition 2.1]{KowalskiSawin2016}, so as in \cite{RicottaRoyer2018} we omit the proof. We contend ourselves with noting that these properties all rest on upper bounds on the coefficients in \eqref{KlpH-truncated}, which are easily uniform in $p$, and that the second estimate in item \eqref{rate-of-convergence-item} follows from the fact that $\Kl(t;p;(a_1,b_0))-\Kl_{H}(t;p;(a_1,b_0))$ is $\sigma_H^2$-sub\-gaussian with $\sigma_H\ll (\sum_{|h|>H}|1/(2\pi i h)|^2)^{1/2}\ll H^{-1/2}$. The variance $\mathbb{V}(\Kl(t;p;(a_1,b_0)))$ may be explicated using the Plancherel formula, though this is not needed for our purposes.

\subsection{Renormalization of Kloosterman paths and completion}
In this section, we define a slight renormalization of Kloosterman paths $\widetilde{\Kl_{p^n}}(t;(a,b_0)):[0,1]\to\mathbb{C}$, which is better suited to the completion techniques. As the two main lemmata of this section, we express $\widetilde{\Kl_{p^n}}(t;(a,b_0))$ in terms of complete Kloosterman sums \eqref{kloost-def}, and we show that the complex moments of $\widetilde{\Kl_{p^n}}(t;(a,b_0))$ are very close to those of the original paths $\Kl_{p^n}(t;(a,b_0))$.

Specifically, define $\widetilde{\text{Kl}_{p^n}}(t;(a,b_0)):[0,1]\to \mathbb C$ as follows: for any $k\in\{1,2,...,p^{n-1}\}$,
\begin{equation}
\label{modified-path}
\forall t\in\Big(\frac{k-1}{p^{n-1}},\frac{k}{p^{n-1}}\Big],\quad \widetilde{\Kl_{p^n}}(t;(a,b_0))=\frac{1}{p^{n/2}}\mathop{\sum\nolimits^\times}_{1\leqslant x\leqslant x_k(t)} e_{p^n}(ax+b_0\overline{x}),
\end{equation}
where
\[ x_k(t)=\varphi(p^{n})t+k-1. \]
The corresponding complex moments are
\begin{equation}
\label{complex-moment-renorm-def}
\widetilde{\mathcal{M}_{p^n}}(\bm{t};\bm{m},\bm{n};b_0)=\frac{1}{p^{n-1}}\sum_{\substack{a\in (\mathbb Z/p^n\mathbb Z)^{\times}\\ a\equiv a_1 \bmod p}}\prod_{i=1}^k \overline{\widetilde{\Kl_{p^n}}(t_i;(a,b_0))}^{m_i}\widetilde{\Kl_{p^n}}(t_i;(a,b_0))^{n_i}.
\end{equation}
Then the two main results about the paths $\widetilde{\Kl}(t;(a,b_0))$ are as follows.

\begin{lemma}[Completion]
\label{completion-lemma}
The modified Kloosterman paths $\widetilde{\Kl}_{p^n}(t;(a,b_0))$ defined in \eqref{modified-path} satisfy
\begin{equation}
\label{completion-eq1}
\widetilde{\Kl_{p^n}}(t;(a,b_0))=\frac{1}{p^{n/2}}\mathop{\sum\nolimits}_{\substack{h\bmod{p^n}\\ (a-h)b_0\in (\mathbb Z/p\mathbb Z)^{\times 2}}} \alpha_{p^n}(h;t)\Kl_{p^n}(a-h;b_0),
\end{equation}
with certain coefficients $\alpha_{p^n}(h;t)$, which are shown in \eqref{alphas-def} and satisfy, for $h\in\mathbb{Z}$ with $|h|<\frac12p^n$,
\begin{equation}
\label{completion-eq2}
\begin{aligned}
&\frac{1}{p^{n/2}}\alpha_{p^n}(h;t)\leqslant\min\Big(1,\frac1{2|h|}\Big);\\
&\frac{1}{p^{n/2}}\alpha_{p^n}(h;t)=\beta(h;t)+\mathrm{O}\Big(\frac{1}{p^n}\Big),\quad
\beta(h;t)=\begin{cases} t, & \text{if }h=0,\\ \frac{e(ht)-1}{2\pi ih} &\text{otherwise}.\end{cases}
\end{aligned}
\end{equation}
\end{lemma}

\begin{lemma}[Approximation of complex moments]
\label{moments-approx-lemma}
The complex moments defined in \eqref{complex-moment-def} and \eqref{complex-moment-renorm-def} satisfy
\[ \mathcal{M}_{p^n}(\bm{t};\bm{m},\bm{n};b_0)=\widetilde{\mathcal{M}_{p^n}}(\bm{t};\bm{m},\bm{n};b_0)+\mathrm{O}_{\ell(\bm{m}+\bm{n})}\big(p^{-n/2}\log^{\ell(\bm{m}+\bm{n})}(p^n)\big). \]
\end{lemma}

The lemmas are mostly analogous to \cite[Lemmata 4.2 and 4.4]{RicottaRoyer2018}, so we will be brief, emphasizing only the new aspects. The classical completion method (essentially the Plancherel identity for the discrete Fourier transform) yields \eqref{completion-eq1} with coefficients $\alpha_{p^n}(h;t)$ defined for $h\in \mathbb Z/p^n\mathbb Z$ and $1\leqslant k\leqslant p^{n-1}$ as
\begin{equation}
\label{alphas-def}
\forall t\in\Big(\frac{k-1}{p^{n-1}},\frac{k}{p^{n-1}}\Big], \quad \alpha_{p^n}(h;t)=\frac{1}{p^{n/2}}\sum_{1\leqslant x\leqslant x_k(t)}e_{p^n}(hx),
\end{equation}
and keeping in mind that $\Kl_{p^n}(a-h,b_0)=0$ if $(a-h)b_0\not\in(\mathbb{Z}/p\mathbb{Z})^{\times 2}$ by Lemma~\ref{kloost-eval}. Executing the geometric sum in $\alpha_{p^n}(h;t)$ we obtain \eqref{completion-eq2} for $|h|<\frac12p^n$, a condition we remark should also be required in \cite[Eq.~(4.7)]{RicottaRoyer2018}. This proves Lemma~\ref{completion-lemma}.

The estimate \eqref{completion-eq2} implies that
 \[ \sum_{h\bmod p^n}|\alpha_{p^n}(h,t)|\ll p^{n/2}\log(p^n). \]
 Coupling this with the bound $|\Kl_{p^n}(a-h;b_0)|\leqslant 2$ from Lemma~\ref{kloost-eval}, we conclude from \eqref{completion-eq1} that
\[ |\widetilde{\Kl_{p^n}}(t;(a,b_0))|\ll \log(p^n). \]
On the other hand, we have that
\begin{equation}
\label{tilde-to-no-tilde}
\big|\text{Kl}_{p^n}(t;(a,b_0))-\widetilde{\text{Kl}_{p^n}}(t;(a,b_0))\big|\leqslant \frac{p}{p^{n/2}}
\end{equation}
by simply counting the number of terms different between the two sums and estimating them trivially. Lemma~\ref{moments-approx-lemma} follows using this by subtracting the corresponding terms in \eqref{complex-moment-def} and \eqref{complex-moment-renorm-def}.

Starting from the definition \eqref{complex-moment-renorm-def} of the complex moments of $\widetilde{\Kl_{p^n}}(t_i;(a,b_0))$, inserting the completion expansion \eqref{completion-eq1}, and expanding, we find that
\begin{equation}
\label{complex-moment-with-sums-of-products}
\begin{aligned}
\widetilde{\mathcal{M}_{p^n}}(\bm{t};\bm{m},\bm{n};b_0)
=\frac{1}{p^{n\ell(\bm{m}+\bm{n})/2}}&\sum_{\bm{h}\in H^{\ell(\bm{m}+\bm{n})}_{a_1(p^n)}}\prod_{i=1}^k\prod_{j=1}^{m_i}\overline{\alpha_{p^n}(h_{i,j};t_i)}\prod_{j=m_i+1}^{m_i+n_i}\alpha_{p^n}(h_{i,j};t_i)\\
&\qquad\times\frac{1}{p^{n-1}}\sum_{\substack{a\in(\mathbb Z/p^n\mathbb Z)^\times\\a\equiv a_1\bmod p}}\prod_{i=1}^k\prod_{j=1}^{m_i+n_i}\Kl_{p^n}(a-h_{i,j},b_0),
\end{aligned}
\end{equation}
where
\[ H_{a_1(p^n)}=\big\{h\in\mathbb{Z}:|h|<\tfrac12p^n,\,\,(a_1-h)b_0\in(\mathbb{Z}/p\mathbb{Z})^{\times 2}\big\}, \]
$\bm{h}_i=(h_{i,1},\dots,h_{i,m_i},h_{i,m_i+1},\dots,h_{i,m_i+n_i})\in H_{a_1(p^n)}^{m_i+n_i}$, and $\bm{h}=(\bm{h}_1,\dots,\bm{h}_k)\in H^{\ell(\bm{m}+\bm{n})}_{a_1(p^n)}$.

Expansion \eqref{complex-moment-with-sums-of-products} makes clear the central importance of the sums of products of Kloosterman sums for the estimation of $\widetilde{M}_{p^n}(\bm{t};\bm{m};\bm{n};b_0)$. We devote the next subsection to these, and return to the analysis of $\widetilde{M}_{p^n}(\bm{t};\bm{m};\bm{n};b_0)$ in \S\ref{main-error-terms}.

\subsection{Sums of products of Kloosterman sums}
The inner sum in \eqref{complex-moment-with-sums-of-products} is a \emph{sum of products} of Kloosterman sums, or equivalently a moment of shifted Kloosterman sums, which for $\bm{\mu}:\mathbb{Z}/p^n\mathbb{Z}\to\mathbb{Z}_{\geqslant 0}$ we denote more generally by
\begin{equation}
\label{sum-of-products}
\mathcal{S}_{p^n}(\bm{\mu};a_1,b_0)=\frac{1}{p^{n-1}}\sum_{\substack{a\in (\mathbb Z/p^n \mathbb Z)^\times\\ a\equiv a_1\bmod p}} \prod_{\tau\in \mathbb Z/p^n\mathbb Z}\text{Kl}_{p^n}(a-\tau,b_0)^{\bm{\mu}(\tau)}.
\end{equation}
Denoting
\[ T(\bm{\mu})=\{\tau\in\mathbb{Z}/p^n\mathbb{Z}:\bm{\mu}(\tau)\geqslant 1\}, \]
and applying the explicit evaluation of Kloosterman sums in Lemma~\ref{kloost-eval}, we get that
\[ \mathcal{S}_{p^n}(\bm{\mu};a_1,b_0)=\frac1{p^{n-1}}\sum_{\substack{a\in(\mathbb{Z}/p^n\mathbb{Z})^{\times}\\a\equiv a_1\bmod p}}\prod_{\tau\in T(\bm{\mu})}\!\!\left(\!\left(\frac{((a-\tau)b_0)_{1/2}}{p^n}\right)\mathrm{Re}\left[\epsilon_{p^n}e\!\left(\frac{2((a-\tau)b_0)_{1/2}}{p^n}\right)\right]\right)^{\!\bm{\mu}(\tau)} \]
if $(a_1-\tau)b_0\in(\mathbb{Z}/p\mathbb{Z})^{\times 2}$ for every $\tau\in T(\bm{\mu})$, and $\mathcal{S}_{p^n}(\bm{\mu};a_1,b_0)=0$ otherwise. Expanding the real parts, we may write
\begin{align*}
\mathcal{S}_{p^n}(\bm{\mu};a_1,b_0)&=\frac1{2^{\|\bm{\mu}\|_1}}\frac1{p^{n-1}}\sum_{\substack{a\in(\mathbb{Z}/p^n\mathbb{Z})^{\times}\\a\equiv a_1\bmod p}}\prod_{\tau\in T(\bm{\mu})}\left(\frac{((a-\tau)b_0)_{1/2}}{p^n}\right)^{\bm{\mu}(\tau)}\\
&\qquad\times
\sum_{u_{\tau}=0}^{\bm{\mu}(\tau)}\binom{\bm{\mu}(\tau)}{u_{\tau}}\epsilon_{p^n}^{\bm{\mu}(\tau)-2u_{\tau}}e\left(\frac{2(\bm{\mu}(\tau)-2u_{\tau})((a-\tau)b_0)_{1/2}}{p^n}\right).
\end{align*}
Denoting, for every $T\subseteq\mathbb{Z}/p^n\mathbb{Z}$ such that $(a_1-\tau)b_0\in(\mathbb{Z}/p\mathbb{Z})^{\times 2}$ for every $\tau\in T$, and for every sequence $\bm{\epsilon}=(\epsilon_{\tau})_{\tau\in T}$,
\begin{equation}
\label{sum-of-exponentials}
f_{T,\bm{\epsilon}}(a)=\sum_{\tau\in T}\epsilon_{\tau}((a-\tau)b_0)_{1/2},\quad
\mathcal{S}^{T,\bm{\epsilon}}_{p^n}(a_1,b_0)=\frac1{p^{n-1}}\sum_{\substack{a\in(\mathbb{Z}/p^n\mathbb{Z})^{\times}\\a\equiv a_1\bmod p}}
e\left(\frac{f_{T,\epsilon}(a)}{p^n}\right),
\end{equation}
we may write
\begin{equation}
\label{spn-big-sum}
\mathcal{S}_{p^n}(\bm{\mu};a_1,b_0)=\sum_{\bm{u}\in U(\bm{\mu})} c(\bm{\mu},\bm{u};a_1,b_0) \mathcal{S}_{p^n}^{T(\bm{\mu}),\bm{\mu}-2\bm{u}}(a_1,b_0),
\end{equation}
where $U(\bm{\mu})=\prod_{\tau\in T(\bm{\mu})}[0,\bm{\mu}(\tau)]$ and, for $\bm{u}\in U(\mu)$,
\[ c(\bm{\mu},\bm{u};a_1,b_0)=\frac1{2^{\|\bm{\mu}\|_1}}\prod_{\tau\in T(\bm{\mu})}\left(\frac{((a_1-\tau)b_0)_{1/2}}{p^n}\right)^{\bm{\mu}(\tau)}\binom{\bm{\mu}(\tau)}{u_{\tau}}\epsilon_{p^n}^{\bm{\mu}(\tau)-2u_{\tau}}. \]

Estimation of the sums $\mathcal{S}_{p^n}^{T,\bm{\epsilon}}(a_1,b_0)$ shown in \eqref{sum-of-exponentials} is the main subject of section~\ref{sums-of-products-sec}. It will be seen there (Theorem~\ref{sums-of-products-theorem}) that, at least when there are no high collusions among the elements of the set of shifts $T$, these sums feature power cancellation unless $\bm{\epsilon}=\bm{0}$ (or equivalently if we can take $T=\emptyset$), which corresponds to the term with $\bm{\mu}=2\bm{u}$ in \eqref{spn-big-sum}, if $2\mid\bm{\mu}$ so that such a term is in fact present. Specifically, applying Theorem~\ref{sums-of-products-theorem} to \eqref{spn-big-sum} we obtain the following.

\begin{theorem}[Moments of shifted Kloosterman sums]
\label{kloosterman-moments}
For every $M>0$, there exist constants $\delta_i=\delta_i(M)$ ($i=1,2$) such that, for every $\bm{\mu}:\mathbb{Z}/p^n\mathbb{Z}\to\mathbb{Z}_{\geqslant 0}$ such that $\|\bm{\mu}\|_{\infty}\leqslant M$ and $\tau\neq\tau'\pmod{p^{\lfloor\delta_2n\rfloor}}$ for every $\tau\neq\tau'\in T(\bm{\mu})$, the normalized moment of shifted Kloosterman sums, given for $a_1,b_0\in(\mathbb{Z}/p\mathbb{Z})^{\times}$ by \eqref{sum-of-products}, satisfies
\[ \mathcal{S}_{p^n}(\bm{\mu};a_1,b_0)=\prod_{\bm{\tau}\in T(\mu)}\delta_{2\mid\bm{\mu}(\tau)}\frac1{2^{\bm{\mu}(\tau)}}\binom{\bm{\mu}(\tau)}{\bm{\mu}(\tau)/2}+\mathrm{O}_{\|\bm{\mu}\|_1}(p^{-\delta_1n}). \]
\end{theorem}

\begin{remark}
Recalling \eqref{moments-of-mu}, from Theorem~\ref{kloosterman-moments} it follows in particular that, for $a_1b_0\in(\mathbb{Z}/p\mathbb{Z})^{\times 2}$, and for every $m\in\mathbb{Z}_{\geqslant 0}$, there exists a $\delta>0$ such that, for every $n\geqslant 2$,
\begin{equation}
\label{corr-one-sum}
\frac1{p^{n-1}}\sum_{\substack{a\in (\mathbb{Z}/p^n\mathbb{Z})^{\times}\\ a\equiv a_1\bmod p}}\Kl_{p^n}(a,b_0)^m=\mathbb{E}(U^{\sharp m})+\mathrm{O}_m(p^{-\delta n}),
\end{equation}
where $U^{\sharp}$ is a random variable distributed according to the probability law $\mu$ in \eqref{def-mu}.
In other words, the ensemble of normalized Kloosterman sums $\Kl_{p^n}(a,b_0)$ with $a\in(\mathbb{Z}/p^n\mathbb{Z})^{\times}_{a_1}$ is equidistributed in $[-2,2]$ with respect to $\mu$ as $n\to\infty$. This statement should be contrasted with the analogous statement \cite[Remark~4.11]{RicottaRoyer2018} for $n$ fixed and $p\to\infty$ with the measure $\mu_U$ shown in \eqref{KKl}; it also intuitively explains the shape of the limiting random series \eqref{our-random-series} as compared to \eqref{after-completion}. We refer to \cite[Remark~4.11]{RicottaRoyer2018} for other instances in which equidistribution with respect to $\mu$ and $\mu_U$ has been noted. In fact, in this case, the analogue of \eqref{corr-one-sum} holds even in substantially smaller orbits $a\equiv a_1\pmod{p^{\kappa}}$ with any $\kappa\leqslant n-2$, with a substantially stronger error term $\mathrm{O}_m(p^{-(n-\kappa)/2})$, because \eqref{corr-one-sum} encounters no singular stationary points in the situation of Theorem~\ref{sums-of-products-theorem}.

More generally, for any fixed finite set $T\subseteq\mathbb{Z}$ such that $(a_1-\tau)b_0)\in(\mathbb{Z}/p\mathbb{Z})^{\times 2}$ for every $\tau\in T$, there exists a $\delta=\delta(T)>0$ such that, for every tuple $\bm{m}=(m_{\tau})_{\tau\in T}\in\mathbb{Z}_{\geqslant 0}^T$ and every $n\geqslant n_0(T,\|\bm{m}\|_1)$,
\[ \frac1{p^{n-1}}\sum_{\substack{a\in (\mathbb{Z}/p^n\mathbb{Z})^{\times}\\ a\equiv a_1\bmod p}}\prod_{\tau\in T}\Kl_{p^n}(a-\tau,b_0)^{m_{\tau}}=\mathbb{E}\bigg(\prod_{\tau\in T}U_{\tau}^{\sharp m_{\tau}}\bigg)+\mathrm{O}_m(p^{-\delta n}), \]
for any sequence of independent random variables $U_{\tau}^{\sharp}$ identically distributed according to $\mu$; cf.~\eqref{expectation-of-products}. Thus the tuple $(\Kl_{p^n}(a-\tau,b_0))_{\tau\in T}$ is equidistributed in $[-2,2]^{|T|}$ with respect to $\otimes^{|T|}\mu$.
\end{remark}

\subsection{Isolation of the main and error terms}
\label{main-error-terms}

We now return to \eqref{complex-moment-with-sums-of-products} and denote, for $\bm{h}\in H_{a_1(p^n)}^{\ell(\bm{m}+\bm{n})}$, by $\bm{\mu}_{\bm{h}}:\mathbb{Z}/p^n\mathbb{Z}\to\mathbb{Z}_{\geqslant 0}$ the function defined by
$\mu_{\bm{h}}(\tau)=\{(i,j):1\leqslant i\leqslant k,\,1\leqslant j\leqslant m_i+n_i,\,h_{i,j}\equiv \tau \bmod p^n\}|$, so that the inner sum in \eqref{complex-moment-with-sums-of-products} is precisely $\mathcal{S}_{p^n}(\bm{\mu}_{\bm{h}};a_1,b_0)$.
Using the decomposition \eqref{spn-big-sum} and isolating the term with $\bm{\mu}_{\bm{h}}=2\bm{u}$ (if present) as in Theorem~\ref{kloosterman-moments}, we can write
\begin{equation}
\label{main-error-isolated}
\widetilde{\mathcal{M}_{p^n}}(\bm{t};\bm{m},\bm{n};b_0)=\widetilde{\mathcal{M}_{p^n}}{}'(\bm{t};\bm{m},\bm{n};b_0)
+\mathrm{Err}_{p^n}(\bm{t};\bm{m},\bm{n};b_0),
\end{equation}
where
\begin{align}
\nonumber
\widetilde{\mathcal{M}_{p^n}}{}'(\bm{t};\bm{m},\bm{n};b_0)
&=\frac{1}{p^{n\ell(\bm{m}+\bm{n})/2}}\sum_{\bm{h}\in H^{\ell(\bm{m}+\bm{n})}_{a_1(p^n)}}\prod_{i=1}^k\prod_{j=1}^{m_i}\overline{\alpha_{p^n}(h_{i,j};t_i)}\prod_{j=m_i+1}^{m_i+n_i}\alpha_{p^n}(h_{i,j};t_i)\\
\nonumber
&\qquad\times\prod_{\tau\in T(\bm{\mu}_{\bm{h}})}\delta_{2\mid\bm{\mu}_{\bm{h}}(\tau)}\frac1{2^{\bm{\mu}_{\bm{h}}(\tau)}}\binom{\bm{\mu}_{\bm{h}}(\tau)}{\bm{\mu}_{\bm{h}}(\tau)/2},\\
\mathrm{Err}_{p^n}(\bm{t};\bm{m},\bm{n};b_0)
\label{error-term}
&=\frac{1}{p^{n\ell(\bm{m}+\bm{n})/2}}\sum_{\bm{h}\in H^{\ell(\bm{m}+\bm{n})}_{a_1(p^n)}}\prod_{i=1}^k\prod_{j=1}^{m_i}\overline{\alpha_{p^n}(h_{i,j};t_i)}\prod_{j=m_i+1}^{m_i+n_i}\alpha_{p^n}(h_{i,j};t_i)\\
\nonumber
&\qquad\times\sum_{\substack{\bm{u}\in U(\bm{\mu}_{\bm{h}})\\\bm{\mu}_{\bm{h}}\neq 2\bm{u}}} c(\bm{\mu},\bm{u};a_1,b_0) \mathcal{S}_{p^n}^{T(\bm{\mu}),\bm{\mu}_{\bm{h}}-2\bm{u}}(a_1,b_0),
\end{align}
Using \eqref{expectation-of-products} and \eqref{completion-eq2}, we find that $\widetilde{\mathcal{M}_{p^n}}{}'(\bm{t};\bm{m},\bm{n};b_0)$ equals
\begin{equation}
\label{main-term-eval}
\begin{aligned}
&\frac{1}{p^{n\ell(\bm{m}+\bm{n})/2}}\sum_{\bm{h}\in H^{\ell(\bm{m}+\bm{n})}_{a_1(p^n)}}\prod_{i=1}^k\prod_{j=1}^{m_i}\overline{\alpha_{p^n}(h_{i,j};t_i)}\prod_{j=m_i+1}^{m_i+n_i}\alpha_{p^n}(h_{i,j};t_i)
\mathbb E\bigg(\prod_{i=1}^k\prod_{j=1}^{m_i+n_i}U_{h_{i,j}}^{\sharp}\bigg)\\
&\,=\sum_{\bm{h}\in H^{\ell(\bm{m}+\bm{n})}_{a_1(p^n)}}\prod_{i=1}^k\prod_{j=1}^{m_i}\overline{\beta(h_{i,j};t_i)}\prod_{j=m_i+1}^{m_i+n_i}\beta(h_{i,j};t_i)\mathbb E\bigg(\prod_{i=1}^k\prod_{j=1}^{m_i+n_i}U_{h_{i,j}}^{\sharp}\bigg)\\
&\,\qquad\qquad+\mathrm{O}_{\ell(\bm{m}+\bm{n})}\bigg(\frac{\log^{\ell(\bm{m}+\bm{n})}p}{p^n}\bigg)\\
&\,=\mathbb E\bigg(\!\prod_{i=1}^k\overline{\Kl_{\frac{p^n-1}{2}}(t_i;p;(a_1,b_0))}^{m_i}\Kl_{\frac{p^n-1}{2}}(t_i;p;(a_1,b_0))^{n_i}\!\bigg)\!\!+\!\mathrm{O}_{\ell(\bm{m}+\bm{n})}\bigg(\frac{\log^{\ell(\bm{m}+\bm{n})}p}{p^n}\bigg)\\
&\,=\mathbb E\Big(\prod_{i=1}^k\overline{\Kl(t_i;p;(a_1,b_0))}^{m_i}\Kl(t_i;p;(a_1,b_0))^{n_i}\Big)\!+\mathrm{O}_{\ell(\bm{m}+\bm{n})}\bigg(\frac{\log^{\ell(\bm{m}+\bm{n})}p}{p^{n/2}}\bigg),
\end{aligned}
\end{equation}
recalling the definition \eqref{KlpH-truncated} and using Proposition~\ref{random-series-properties} in the final two equalities.

Combining \eqref{main-error-isolated}, \eqref{main-term-eval}, and \eqref{error-term}, we obtain the following.

\begin{proposition}
\label{main-error-terms-proposition}
The complex moment $\widetilde{\mathcal{M}_{p^n}}(\bm{t};\bm{m},\bm{n};b_0)$ defined in \eqref{complex-moment-renorm-def} satisfies
\begin{align*}
\widetilde{\mathcal{M}_{p^n}}(\bm{t};\bm{m},\bm{n};b_0)
&=\mathbb E\Big(\prod_{i=1}^k\overline{\Kl(t_i;p;(a_1,b_0))}^{m_i}\Kl(t_i;p;(a_1,b_0))^{n_i}\Big)\\
&\qquad\qquad+\mathrm{Err}_{p^n}(\bm{t};\bm{m};\bm{n};b_0)
+\mathrm{O}_{\ell(\bm{m}+\bm{n})}\bigg(\frac{\log^{\ell(\bm{m}+\bm{n})}p}{p^{n/2}}\bigg),
\end{align*}
where the error term $\mathrm{Err}_{p^n}(\bm{t};\bm{m};\bm{n};b_0)$ is defined in \eqref{error-term}.
\end{proposition}

\subsection{Estimation of error terms}
\label{error-terms-sec}
In this section, we address the term $\mathrm{Err}_{p^n}(\bm{t};\bm{m};\bm{n};b_0)$, and use the estimates on complete exponential sums from Section~\ref{sums-of-products-sec} to prove the following estimate.

\begin{proposition}
\label{error-term-estimate-proposition}
There exists a constant $\delta=\delta(\|\bm{m}\|_1+\|\bm{n}\|_1)>0$ such that the error term $\mathrm{Err}_{p^n}(\bm{t};\bm{m};\bm{n};b_0)$ defined in \eqref{error-term} satisfies
\[ \mathrm{Err}_{p^n}(\bm{t};\bm{m};\bm{n};b_0)\ll_{\ell(\bm{m}+\bm{n})} p^{-\delta n}. \]
\end{proposition}

\begin{proof}
For every $\bm{h}\in H_{a_1(p^n)}^{\ell(\bm{m}+\bm{n})}$, let
\[ \Delta(\bm{h})=\min\big\{\|h_{i,j}-h_{i',j'}\|_p:(i,j)\neq(i',j')\big\}, \]
where here and later in the proof all $(i,j)$ range through the set of indices $\{1\leqslant i\leqslant k,\,1\leqslant j\leqslant m_i+n_i\}$. Then, Theorem~\ref{sums-of-products-theorem} shows that the sums $\mathcal{S}_{p^n}^{T(\bm{\mu}_{\bm{h}},\bm{\mu}_{\bm{h}}-2\bm{u})}(a_1,b_0)$ appearing in $\mathrm{Err}_{p^n}(\bm{t};\bm{m};\bm{n};b_0)$ exhibit power cancellation for $\Delta(\bm{h})>p^{-\delta_2n}$. With this in mind, we decompose
\begin{align*}
\mathrm{Err}_{p^n}(\bm{t};\bm{m};\bm{n};b_0)
&=\mathrm{Err}_{p^n}(\bm{t};\bm{m};\bm{n};b_0)^{\circ}+ \mathrm{Err}_{p^n}(\bm{t};\bm{m};\bm{n};b_0)^{\mathrm{sing}}\\&:=
\frac{1}{p^{n\ell(\bm{m}+\bm{n})/2}}\sum_{\substack{\bm{h}\in H^{\ell(\bm{m}+\bm{n})}_{a_1(p^n)}\\\Delta(\bm{h})> p^{-\delta_2n}}}\cdots+\frac{1}{p^{n\ell(\bm{m}+\bm{n})/2}}\sum_{\substack{\bm{h}\in H^{\ell(\bm{m}+\bm{n})}_{a_1(p^n)}\\\Delta(\bm{h})\leqslant p^{-\delta_2n}}}\cdots,
\end{align*}
with terms in both summations exactly as in \eqref{error-term}.

Using Theorem~\ref{sums-of-products-theorem} and the estimate \eqref{completion-eq2} of Lemma~\ref{completion-lemma}, we can estimate
\begin{align*}
\mathrm{Err}_{p^n}(\bm{t};\bm{m};\bm{n};b_0)^{\circ}
&\ll_{\ell(\bm{m}+\bm{n})}\sum_{\bm{h}\in H^{\ell(\bm{m}+\bm{n})}_{a_1(p^n)}}
\frac1{\prod_{i=1}^k\prod_{j=1}^{m_i+n_i}(|h_{i,j}|+1)} p^{-\delta_1n}\\
&\ll_{\ell(\bm{m}+\bm{n})}\log^{\ell(\bm{m}+\bm{n})}(p^n)p^{-\delta_1 n},
\end{align*}
which is clearly acceptable. As for $\mathrm{Err}_{p^n}(\bm{t};\bm{m};\bm{n};b_0)^{\mathrm{sing}}$, we begin by noting the simple estimate that, for every $h\in\mathbb{Z}$,
\[ \sum_{\substack{\{h'\neq h:\\M\geqslant |h'|\geqslant |h|,\\h'\equiv h\bmod p^r\}}}\frac1{|h|+1}\ll\frac{\log M}{p^r}, \]
simply by writing $h'=h+p^rl$ with $l\in\mathbb{Z}$. Using this we can estimate
\begin{align*}
&\mathrm{Err}_{p^n}(\bm{t};\bm{m};\bm{n};b_0)^{\mathrm{sing}}\\
&\qquad\ll_{\ell(\bm{m}+\bm{n})}\sum_{(i_0,j_0)\neq(i',j')}\sum_{\substack{\bm{h}\in H^{\ell(\bm{m}+\bm{n})}_{a_1(p^n)}\\ h_{i',j'}\neq h_{i_0,j_0}\\ |h_{i',j'}|\geqslant |h_{i_0,j_0}|\\ h_{i',j'}\equiv h_{i_0,j_0}\bmod{p^{\lfloor\delta_2n\rfloor}}}}
\frac1{\mathop{\prod_{i=1}^k\prod_{j=1}^{m_i+n_i}}\limits_{\substack{(i,j)\neq (i',j')}}(|h_{i,j}|+1)}\frac1{|h_{i',j'}|+1} \\
&\qquad\ll_{\ell(\bm{m}+\bm{n})}\log^{\ell(\bm{m}+\bm{n})}(p^n)p^{-\delta_2 n},
\end{align*}
which is also acceptable. Putting everything together proves the proposition.
\end{proof}

\section{Complete exponential sums}
\label{sums-of-products-sec}

\subsection{Hensel's lemma}
In this section, we prove a version of Hensel's lemma for roots of a congruence of the form
\[ f(x)\equiv 0\pmod{p^n}, \]
where, for some domain $X\subseteq\mathbb{Z}/p^n\mathbb{Z}$, $f:X\to\mathbb{Z}/p^n\mathbb{Z}$ is a function satisfying a condition resembling differentiability (in particular, $f$ could arise from an honestly analytic function on a domain inside $\mathbb{Z}_p$), and roots may be singular in a controlled fashion.

Importantly we remove the condition that $f$ is a polynomial, which is used in, say, \cite[Theorem 2.24]{NivenZuckermanMontgomery1991}. Of course, since we are talking about functions on $X\subseteq\mathbb{Z}/p^n\mathbb{Z}$, they can automatically formally be expressed as polynomials; however, these are typically of very high degree, which can be deadly in applications (cf.~\cite{RicottaRoyer2018}). More properly speaking we remove the requirement to \emph{think} of $f$ as a polynomial; rather, we typically think of $f$ as a reduction of a fixed $p$-adically analytic function on a subdomain of $\mathbb{Z}_p$.

\begin{lemma}[Hensel's lemma with singular roots]
\label{hensel-lemma}
Let $n\geqslant\kappa_0\geqslant 1$, let a domain $X\subseteq\mathbb{Z}/p^n\mathbb{Z}$ be $(p^{\kappa_0}\mathbb{Z}/p^n\mathbb{Z})$-translation invariant, and let $f,f_1:X\to\mathbb{Z}/p^n\mathbb{Z}$ be functions such that
\begin{equation}
\label{diffble-cond}
f(x+p^{\kappa}t)\equiv f(x)+f_1(x)p^{\kappa}t\pmod{p^{\min(2\kappa,n)}}\quad (x\in X,\,\,t\in\mathbb{Z}/p^n\mathbb{Z},\,\,\kappa\geqslant\kappa_0).
\end{equation}
Suppose that $a\in X$ satisfies for some $0\leqslant\rho<n$
\[ f(a)\equiv 0\pmod{p^j},\quad p^{\rho}\exmid f_1(a),\quad j\geqslant \min(2\rho+1,\rho+\kappa_0). \]
Then:
\begin{enumerate}
\item\label{hensel1} If $b\in X$ satisfies $b\equiv a\pmod{p^{j-\rho}}$, then $f(b)\equiv 0\pmod{p^j}$ and $p^{\rho}\exmid f_1(b)$.
\item\label{hensel2} There is a unique $t$ modulo $p$ such that $f(a+tp^{j-\rho})\equiv 0\pmod{p^{j+1}}$.
\item\label{hensel3} There is a unique $t$ modulo $p^{n-j}$ such that
\[ f(a+tp^{j-\rho})\equiv 0\pmod{p^n}. \]
\end{enumerate}
\end{lemma}

\begin{proof}
Along the similar lines as in \cite[Theorem 2.24]{NivenZuckermanMontgomery1991}, since $2j-2\rho\geqslant j+1$, using \eqref{diffble-cond} we have for $b=a+tp^{j-\rho}\in X$ that
\begin{equation}
\label{diffble-applied}
f(b)=f(a+tp^{j-\rho})\equiv f(a)+f_1(a)tp^{j-\rho}\pmod{p^{j+1}}.
\end{equation}
In particular, since the right-hand side is divisible by $p^j$, so is the left side, and $f(b)\equiv 0\pmod{p^j}$. Moreover using \eqref{diffble-cond} with the roles of $a$ and $b$ switched, we find that
\begin{equation}
\label{exmid-proof}
f_1(a)tp^{j-\rho}\equiv f(b)-f(a)\equiv -f_1(b)(-tp^{j-\rho})\pmod{p^{j+1}},
\end{equation}
from which it follows in particular that $p^{\rho}\exmid f_1(b)$. This is immediate from \eqref{exmid-proof} if $p\nmid t$, anf otherwise follows by first arguing that $p^{\rho}\exmid f_1(a+p^{j-\rho})$ and then $p^{\rho}\exmid f_1(b)$. This completes the proof of \eqref{hensel1}.

Dividing both sides of \eqref{diffble-applied} by $p^j$, the above congruence is equivalent to the non-degenerate linear congruence
\[ \frac{f(a+tp^{j-\rho})}{p^j}\equiv\frac{f(a)}{p^j}+t\frac{f_1(a)}{p^{\rho}}\pmod p, \]
which has a unique solution in $t$ modulo $p$. This proves \eqref{hensel2}, and \eqref{hensel3} follows by induction.
\end{proof}

\subsection{Estimates on complete exponential sums}
Let $T\subseteq\mathbb{Z}/p^n\mathbb{Z}$ and $b_0\in(\mathbb{Z}/p^n\mathbb{Z})^{\times}$. Define
\begin{equation}
\label{range-of-definition}
(\mathbb{Z}/p^n\mathbb{Z})^{[T]}=\big\{a\in\mathbb{Z}/p^n\mathbb{Z}:(a-\tau)b_0\in(\mathbb{Z}/p^n\mathbb{Z})^{\times 2}\quad\text{for every }\tau\in T\big\}.
\end{equation}
Further, for every vector $\bm{\epsilon}=(\epsilon_{\tau})_{\tau\in T}\in\mathbb{Z}^T$, we define a function $f_{T,\bm{\epsilon}}:(\mathbb{Z}/p^n\mathbb{Z})^{[T]}\to\mathbb{Z}/p^n\mathbb{Z}$ by
\begin{equation}
\label{f-T-eps}
f_{T,\bm{\epsilon}}(a)=\sum_{\tau\in T}\epsilon_{\tau}((a-\tau)b_0)_{1/2}.
\end{equation}
where $(\cdot)_{1/2}$ is the branch of square root fixed in \S\ref{sqroot-section}. This is the phase appearing in \eqref{sum-of-exponentials}, and hence the central importance of Theorem~\ref{sums-of-products-theorem} for the proof of Theorem~\ref{main}. The principal result of this section is the following theorem.

\begin{theorem}
\label{sums-of-products-theorem}
For every $M>0$, there exist constants $\delta_i=\delta_i(M)>0$ ($i=1,2$) such that for every $n\in\mathbb{N}$, $T\subseteq\mathbb{Z}/p^n\mathbb{Z}$, every set $X\subseteq(\mathbb{Z}/p^n\mathbb{Z})^{[T]}$ as in \eqref{range-of-definition} that is invariant under translations by $p^{\lfloor\delta_2n\rfloor}\mathbb{Z}/p^n\mathbb{Z}$, $\bm{\epsilon}=(\epsilon_{\tau})_{\tau\in T}\in\mathbb{Z}^T$ with $\|\bm{\epsilon}\|_1=\sum_{\tau\in T}|\epsilon_{\tau}|\leqslant M$, and $f_{T,\bm{\epsilon}}$ as in \eqref{f-T-eps}, either
\[ \sum_{a\in X}e\left(\frac{f_{T,\bm{\epsilon}}(a)}{p^n}\right)\ll_{|T|}p^{(1-\delta_1)n} \]
or
\[ \tau\equiv\tau'\pmod{p^{\lfloor\delta_2n\rfloor}} \]
for some $\tau,\tau'\in T$ with $\tau\neq\tau'$ and $\epsilon_{\tau},\epsilon_{\tau'}\neq 0$.
\end{theorem}

The function $f_{T,\bm{\epsilon}}$ shown in \eqref{f-T-eps} may, of course, be understood as the reduction of a properly $p$-adic analytic function on an open domain whose reduction modulo $p^n\mathbb{Z}_p$ is $(\mathbb{Z}/p^n\mathbb{Z})^{[T]}$. Though we will not explicitly use the fact that these are derivatives of $f$, we consider for every $j\in\mathbb{Z}_{\geqslant 0}$ the function $f_{T,\bm{\epsilon}}^{(j)}:(\mathbb{Z}/p^n\mathbb{Z})^{[T]}\to\mathbb{Z}/p^n\mathbb{Z}$ given by
\begin{equation}
\label{derivatives}
f_{T,\bm{\epsilon}}^{(j)}(a)=\sum_{\tau\in T}\epsilon_{\tau}\tbinom{1/2}j((a-\tau)b_0)_{1/2}^{1-2j}.
\end{equation}

\begin{definition}
Let $n\in\mathbb{N}$, $T\subseteq\mathbb{Z}/p^n\mathbb{Z}$, $\bm{\epsilon}\in\mathbb{Z}^T$, and let $f_{T,\bm{\epsilon}}$ be as in \eqref{f-T-eps}. For $J\in\mathbb{N}$, $r\in\mathbb{Z}_{\geqslant 0}$, we say that $a\in (\mathbb{Z}/p^n\mathbb{Z})^{[T]}$ is \emph{$p^r$-singular of order $J$} for $f_{T,\bm{\epsilon}}$ if
\[ f^{(j)}(a)\equiv 0\pmod {p^r}\qquad \text{for every }1\leqslant j\leqslant J, \]
and we denote the set of such residues as $\langle f_{T,\bm{\epsilon}},J,p^r\rangle$.
\end{definition}

The notion of $p^r$-singular solutions of increasing order underlies the following two statements, which are the key propositions of this section.

\begin{proposition}
\label{multiple-stationary-phase}
For every $J\in\mathbb{N}$, there exist constants $\delta_i=\delta_i(J)>0$ ($i=1,2$) such that for every $n\in\mathbb{N}$, $T\subseteq\mathbb{Z}/p^n\mathbb{Z}$, every set $X\subseteq (\mathbb{Z}/p^n\mathbb{Z})^{[T]}$ as in \eqref{range-of-definition} that is invariant under translations by $p^{\lfloor\delta_2n\rfloor}\mathbb{Z}/p^n\mathbb{Z}$, $\bm{\epsilon}\in\mathbb{Z}^T$,
and $f_{T,\bm{\epsilon}}$ as in \eqref{f-T-eps},
\[ \sum_{a\in X}e\left(\frac{f_{T,\bm{\epsilon}}(a)}{p^n}\right)
=\sum_{a\in X\cap\langle f_{T,\bm{\epsilon}},J,p^{\lfloor\delta_2n\rfloor}\rangle}e\left(\frac{f_{T,\bm{\epsilon}}(a)}{p^n}\right)
+\mathrm{O}_J\big(p^{(1-\delta_1)n}\big). \]
\end{proposition}

\begin{proposition}
\label{deep-singular}
For every $M_0,M_1>0$, there exists constants $\delta=\delta(M_0)>0$ and $\varrho=\varrho(M_1)$ such that for every $n\in\mathbb{N}$, $T\subseteq\mathbb{Z}/p^n\mathbb{Z}$ with $|T|\leqslant M_0$, $\bm{\epsilon}\in\mathbb{Z}^T$ with $\|\bm{\epsilon}\|_1\leqslant M_1$, $f_{T,\bm{\epsilon}}$ as in \eqref{f-T-eps}, and $r\in\mathbb{N}$, 
\[ \min\big\{\|\tau-\tau'\|_p:\tau\neq\tau'\in T,\,\epsilon_{\tau},\epsilon_{\tau'}\neq 0\big\}\geqslant p^{-\delta(r-\varrho)}\quad\Rightarrow\quad\langle f_{T,\bm{\epsilon}},|T|,p^r\rangle=\emptyset. \]
\end{proposition}

Theorem~\ref{sums-of-products-theorem} follows directly from Propositions~\ref{multiple-stationary-phase} and \ref{deep-singular}. Proposition~\ref{multiple-stationary-phase} iterates the method of stationary phase (Lemma~\ref{statphase-lemma}) to show that all terms in the complete exponential sums with phase $f_{T,\bm{\epsilon}}$ exhibit power cancellation except possibly for the highly singular terms satisfying an increasing number of congruences to a sizable ower of $p$. Proposition~\ref{deep-singular} uses a completely different argument to show that the latter cannot happen with $|T|$ or more conditions, unless the shifts $\tau\in T$ in \eqref{f-T-eps} themselves collude to a sizable power of $p$.

We note that the proof of Proposition~\ref{multiple-stationary-phase} shows that $\delta_1(J)=\delta_2(J)=\frac1{2^J}$ is allowable. The proof of Proposition~\ref{deep-singular} shows that $\delta(M_0)=\tbinom{M_0}2^{-1}$ is allowable and gives an explicit (fairly mild) dependence of $\varrho=\varrho(M_1)$, cf.~\eqref{varrhos-def}.

\begin{proof}[Proof of Proposition~\ref{multiple-stationary-phase}]
We begin by noting that we may assume that $n$ is sufficiently large, since otherwise the conclusion is vacuous. Next, we note that \eqref{sqroot-diff} implies that the system of functions $f_{T,\epsilon}^{(j)}:(\mathbb{Z}/p^n\mathbb{Z})^{[T]}\to\mathbb{Z}/p^n\mathbb{Z}$ defined in \eqref{derivatives} satisfies \eqref{diffble-cond} in the form
\[ f_{T,\bm{\epsilon}}^{(j)}(x+p^{\kappa}t)\equiv f_{T,\bm{\epsilon}}^{(j)}(x)+f_{T,\bm{\epsilon}}^{(j+1)}(x)\cdot p^{\kappa}t\pmod{p^{\min(2\kappa,n)}} \]
for $x\in (\mathbb{Z}/p^n\mathbb{Z})^{[T]}$, $\kappa\geqslant 1$, and $t\in\mathbb{Z}/p^n\mathbb{Z}$. In particular, for every $J,r\in\mathbb{N}$, the set $\langle f_{T,\bm{\epsilon}},J,p^r\rangle$ is invariant under translations by $p^r\mathbb{Z}/p^n\mathbb{Z}$.

We prove Proposition~\ref{multiple-stationary-phase} by induction on $J$. For $J=1$, this follows immediately (with $\delta_2=\frac12$, and no error term, so $\delta_1=1$) from the stationary phase Lemma~\ref{statphase-lemma}. Suppose the claim is proved from some $J\in\mathbb{N}$, and denote for brevity $r=\lfloor\delta_2n\rfloor$ and $X\langle J,p^r\rangle=X\cap\langle f_{T,\bm{\epsilon}},J,p^{\lfloor\delta_2n\rfloor}\rangle$. In particular, we may assume that $r$ is sufficiently large. Fixing $r'=\lfloor\frac12(r-1)\rfloor$, we have a disjoint union
\begin{align}
\label{X-decomposition}
X\langle J,p^r\rangle&=X\langle J,p^r\rangle^{\flat}\sqcup X\langle J,p^r\rangle^{\sharp},\\
\nonumber X\langle J,p^r\rangle^{\flat}&=\big\{a\in X\langle J,p^r\rangle:p^{\rho}\exmid f_{T,\bm{\epsilon}}^{(J+1)}(a)\text{ for some }0\leqslant\rho\leqslant r'\big\},\\
\nonumber X\langle J,p^r\rangle^{\sharp}&=\big\{a\in X\langle J,p^r\rangle:p^{r'+1}\mid f_{T,\bm{\epsilon}}^{(J+1)}(a)\big\}.
\end{align}
The idea is that the set $X\langle J,p^r\rangle^{\flat}$ collects points $a\in X\langle J,p^r\rangle$ where we have some control on the order of possible singularity of $a$ as a root of $f_{T,\bm{\epsilon}}^{(J)}(x)\equiv 0\pmod{p^r}$, putting us in position to use the possibly singular Hensel's Lemma~\ref{hensel-lemma} to control the number of such~$a$. The remaining set $X\langle J,p^r\rangle^{\flat}$ consists of the more stubbornly deeply singular points $a\in X\langle J,p^r\rangle$; we do not get to control their number, but we collect more information about them.

Specifically, applying Lemma~\ref{hensel-lemma} to the congruence
\[ f_{T,\bm{\epsilon}}^{(J)}(x)\equiv 0\pmod{p^{2\rho+1}} \]
for each $0\leqslant\rho\leqslant r'$ separately, we find that the set $X\langle J,p^r\rangle^{\flat}$ is invariant under translations by $p^r\mathbb{Z}$, and the congruence $f_{T,\bm{\epsilon}}^{(J)}(x)\equiv 0\pmod{p^r}$ has at most $\sum_{\rho=0}^{r'}p^{\rho+1}\ll p^{r'}$ solutions (modulo $p^r$) in $X\langle J,p^r\rangle^{\flat}/p^r\mathbb{Z}$. In particular,
\begin{equation}
\label{flat-estimate}
\big|X\langle J,p^r\rangle^{\flat}\big|\ll p^{n-r/2}.
\end{equation}

On the other hand, of course,
\begin{equation}
\label{sharp-inclusion}
X\langle J,p^r\rangle^{\sharp}\subseteq X\langle J+1,p^{r'}\rangle.
\end{equation}
The inductive step claim for $J+1$ follows from \eqref{X-decomposition}, \eqref{flat-estimate}, and \eqref{sharp-inclusion}.
\end{proof}

\begin{proof}[Proof of Proposition~\ref{deep-singular}]
We prove the proposition by contraposition. Using \eqref{derivatives}, the assumption that $a\in\langle f_{T,\bm{\epsilon}},|T|,p^r\rangle$ may be explicitly written as
\begin{equation}
\label{hidden-matrix}
\sum_{\tau\in T}\tbinom{1/2}j((a-\tau)b_0)^{-(j-1)}\cdot \epsilon_{\tau}((a-\tau)b_0)^{-1}_{1/2}\equiv 0\pmod{p^r} \qquad(1\leqslant j\leqslant |T|).
\end{equation}

Let
\begin{equation}
\label{varrhos-def}
\varrho_1=\max_{1\leqslant j\leqslant |T|}\ord_p\tbinom{1/2}j,\quad \varrho_2=\min_{\tau\in T}\ord_p\epsilon_{\tau}, \quad\varrho=\varrho_1+\varrho_2;
\end{equation}
in particular, it is clear that $\varrho_1$ depends on $|T|$ only and $\varrho_2$ depends on $\|\bm{\epsilon}
\|_{\infty}$ only. Denoting by $A\in(\mathbb{Z}/p^{r-\varrho}\mathbb{Z})^{|T|\times|T|}$ and $\mathbf{v}\in(\mathbb{Z}/p^{r-\varrho_2}\mathbb{Z})^{|T|}$ the matrix and vector given by
\[ A=\Big(((a-\tau)b_0)^{-(j-1)}\Big)_{\tau\in T,\,1\leqslant j\leqslant |T|},\quad \mathbf{v}=\Big(p^{-\varrho_2}\epsilon_{\tau}((a-\tau)b_0)^{-1}_{1/2}\Big)_{\tau\in T}, \]
the condition \eqref{hidden-matrix} implies that
\[ A\mathbf{v}\equiv 0\pmod {p^{r-\varrho}}. \]

Our definition \eqref{varrhos-def} is constructed carefully so that the vector  $\mathbf{v}\in(\mathbb{Z}/p^{n-\varrho_2}\mathbb{Z})^{|T|}$ has at least one unit coordinate $v_{\tau_0}\in(\mathbb{Z}/p^{n-\varrho_2}\mathbb{Z})^{\times}$. Using Gaussian elimination as in the proof of Cramer's rule, which we emphasize can be performed in the ring $\mathbb{Z}/p^{r-\varrho}\mathbb{Z}$ without division by non-units, from this we can conclude that
\[ \det(A)\cdot v_{\tau_0}\equiv 0\pmod{p^{r-\varrho}}, \]
and hence $p^{r-\varrho}\mid\det A$. But the matrix $A$ is a Vandermonde matrix, and so this condition implies that
\[ p^{r-\varrho}\mid\det A\in(\mathbb{Z}/p^{r-\varrho}\mathbb{Z})^{\times}\prod_{\{\tau\neq\tau'\}\subseteq T}(\tau-\tau'), \]
from which we can conclude that, for some $\tau\neq\tau'\in T$,
\[ p^{\lceil \binom{|T|}2^{-1}(r-\varrho)\rceil}\mid(\tau-\tau'). \]
The claim of the proposition follows.
\end{proof}

\begin{remark}
\label{more-general-remark}
The method of this section works in much more generality. It is not hard to see, for example, that it immediately applies to complete sums of arithmetic functions $F:X\to\mathbb{C}$, $X\subseteq\mathbb{Z}/p^n\mathbb{Z}$ satisfying for $\kappa\geqslant\lceil n/2\rceil$ a condition of the form
\[ F(a+p^{\kappa}t)=F(a)e\left(\frac{p^{\kappa}t\cdot f_1(a)}{p^n}\right), \]
where $f_1(a)$ is a linear combination of fixed branches of a set $T$ of power functions (defined on $a_1+p\mathbb{Z}/p^n\mathbb{Z}$ by the power series $s(a_1)(1+p\bar{a}_1x)^{\alpha}$ modulo $p^n$, $\alpha\not\in\{0,1,\dots,|T|-1\}$, with a suitable choice function $s$ similarly to \S\ref{sqroot-section}). For example, via Postnikov's formula (see~\cite[Lemma~13]{Milicevic2016}), $F$ could be the product of finitely many linear shifts of Dirichlet characters modulo $p^n$ or combinations thereof with Kloosterman sums or other summands that themselves arise as sums of certain complete exponential sums modulo $p^n$.
\end{remark}

\section{Convergence in law}
\label{in-law-sec}
In this section, we prove all of our principal statements about convergence in law of random variables induced by Kloosterman paths. Subsections \S\ref{tightness-thm1-subsec} and \ref{proof-of-main} are devoted to the convergence in law statement of Theorem~\ref{main}. We deduce this from convergence in the sense of finite distributions by using Prokhorov's criterion (Proposition~\ref{prokhorov}), which requires tightness and which we in turn verify using Kolmogorov's criterion (Proposition~\ref{kolmogorov}). We follow a similar strategy in \S\ref{rearranged-paths-subsec}, where we prove Theorem~\ref{rearranged-paths-thm} on the limiting distribution of rearranged Kloosterman paths, and in \S\ref{large-p-section}, where in Theorem~\ref{large-p-theorem} we prove the convergence in law of $\Kl^{\bullet}(\cdot;p)\to\Kl$ as $p\to\infty$.

\subsection{Tightness}
\label{tightness-thm1-subsec}
The goal of this section is to prove the following statement, which is a key ingredient in verifying the tightness criterion for convergence in law in Theorem~\ref{main}.
\begin{proposition}
\label{tightness-prop}
Let $p$ be a fixed odd prime, and let $a_1,b_0\in(\mathbb{Z}/p\mathbb{Z})^{\times}$. For every even integer $\alpha\geqslant 4$, there exists a $\beta=\beta(\alpha)>0$ such that, for every $0\leqslant s,t\leqslant 1$,
\[ \frac1{p^{n-1}}\sum_{a\in(\mathbb{Z}/p^n\mathbb{Z})^{\times}_{a_1}}\big|\Kl_{p^n}(t;(a,b_0))-\Kl_{p^n}(s;(a,b_0))\big|^{\alpha}\ll_{\alpha} |t-s|^{1+\beta}. \]
\end{proposition}

Using Kolmogorov's criterion (Proposition~\ref{kolmogorov}), this immediately implies the following.
\begin{corollary}
\label{tight-corollary}
The sequence of $C^0([0,1],\mathbb{C})$-valued random variables $\Kl_{p^n}(\cdot;(a_1,b_0))$ defined in  \eqref{kl-pn-a1-def} is tight as $n\to\infty$.
\end{corollary}

Proposition~\ref{tightness-prop} follows directly from Lemmata~\ref{tight-lemma1}, \ref{tight-lemma2}, and \ref{tight-lemma3} below. A similar tightness proposition was proved in \cite[Theorem 1.2]{RicottaRoyerShparlinski2020}, for a fixed $n\geqslant 31$ and $p\to\infty$, and with a very large $\alpha$; we follow their outline but with substantial adjustments. The core of the argument in \cite{RicottaRoyerShparlinski2020} was in available estimates on incomplete Kloosterman sums of the form \eqref{incomplete-kloost} in an interval $\mathcal{I}$ of length $p^{n(1/2-\lambda)}\leqslant|\mathcal{I}|\leqslant p^{n(1/2+\lambda)}$ for some small $\lambda>0$. As will be seen in the proof of Lemma~\ref{tight-lemma3}, this range of interest is directly related to the strength of the error terms in estimates of sums of products of Kloosterman sums, which are more delicate in the $n\to\infty$ aspect (Theorem~\ref{sums-of-products-theorem}).

The exponential sum core of Proposition~\ref{tightness-prop} is the estimate \eqref{incomplete-kloost}, which shows that incomplete Kloosterman sums over intervals $\mathcal{I}$ as long as $|\mathcal{I}|\leqslant p^{(1-\lambda)n}$ for an arbitrarily small $\lambda>0$ exhibit cancellation better than P\'olya--Vinogradov strength. We deduce this estimate in Proposition~\ref{short-ish-sums} from the machinery of $p$-adic exponent pairs of \cite{Milicevic2016}. In a qualitative form, this proposition is essentially dual to \cite[Theorem 1.2]{LiuShparlinskiZhang2018} via completion or the $B$-process of \cite{Milicevic2016}.

We also sharpen the treatment of the ``middle range'', $1/(\varphi(p^n)-1)\leqslant |t-s|\leqslant p^{-\lambda n}$; our Lemma~\ref{tight-lemma2} shows (crucially for us) that in fact any $\alpha>2$ suffices in this range. We remind the reader that, as in the rest of the paper, all implicit constants are allowed to (polynomially) depend on $p$, so that the estimates are nontrivial only for sufficiently large $n$.

\begin{proposition}
\label{short-ish-sums}
For every $0<\lambda<1$, there exists a $\delta=\delta(\lambda)>0$, depending on $\lambda$ only, such that for every $n\in\mathbb{N}$ and every interval $\mathcal{I}$ of length $|{\mathcal{I}}|\leqslant p^{n(1-\lambda)}$ and every $a,b\in(\mathbb{Z}/p^n\mathbb{Z})^{\times}$,
\begin{equation}
\label{incomplete-kloost}
\frac1{p^{n/2}}\sum_{x\in\mathcal{I}}e_{p^n}(ax+b\bar{x})\ll_{\lambda} p^{-n\delta}.
\end{equation}
\end{proposition}

\begin{proof}
We fix a $k\in\mathbb{N}$ with $k>1/\lambda$, and consider the $p$-adic exponent pair given by \cite[Theorem 2]{Milicevic2016}
\[ BA^kB(0,1)=\Big(\frac{2^{k-1}-k}{2^k-2},\frac{2^{k-1}}{2^k-2}\Big) .\]
Then, according to \cite[Definition~2]{Milicevic2016}, there exists a $\delta'>0$ such that
\[ \sum_{x\in\mathcal{I}}e_{p^n}(ax+b\bar{x})\ll_k (p^n)^{\frac{2^{k-1}-k}{2^k-2}}|\mathcal{I}|^{\frac{k}{2^k-2}}(\log p^n)^{\delta'}. \]
In fact it is clear from \cite[Theorems 4 and 5]{Milicevic2016} that we can take $\delta'=\frac12$, although this is not important for us. Using the condition that $|{\mathcal{I}}|\leqslant p^{n(1-\lambda)}$ and picking any $0<\epsilon<k\lambda-1$, we conclude that
\[ \sum_{x\in\mathcal{I}}e_{p^n}(ax+b\bar{x})\ll_{k,\epsilon}(p^n)^{\frac{2^{k-1}-k+k(1-\lambda)+\epsilon}{2^k-2}}
=p^{n/2-n\delta}, \]
with $\delta=\frac{k\lambda-(1+\epsilon)}{2^k-2}>0$.
\end{proof}

\begin{lemma}[{\cite[Lemma 4.2]{RicottaRoyerShparlinski2020}}]
\label{tight-lemma1}
If $\alpha>0$ and
\[ 0\leqslant |t-s|\leqslant\frac1{\varphi(p^n)-1}, \]
then
\[ \big|\Kl_{p^n}(t;(a,b_0))-\Kl_{p^n}(s;(a,b_0))\big|^{\alpha}\leqslant 2^{\alpha}|t-s|^{\alpha/2}. \]
\end{lemma}

\begin{lemma}[{\cite[Lemma 4.3]{RicottaRoyerShparlinski2020}}]
\label{tight-lemma0}
If $\alpha\geqslant 1$ and
\[ |t-s|\geqslant \frac1{\varphi(p^n)-1}, \]
then
\[ \big|\Kl_{p^n}(t;(a,b_0))-\Kl_{p^n}(s;(a,b_0))\big|^{\alpha}\ll_{\alpha}|t-s|^{\alpha/2}+
\big|\widetilde{\Kl_{p^n}}(t;(a,b_0))-\widetilde{\Kl_{p^n}}(s;(a,b_0))\big|^{\alpha}. \]
\end{lemma}

We import Lemmata~\ref{tight-lemma1} and \ref{tight-lemma0} directly from \cite{RicottaRoyerShparlinski2020} with the same proofs, noting that each holds without an average over $a$, and that, although the statements there ask for $n\geqslant 2$ to be fixed, this is never used in the proofs. Indeed, in this very short regime for $|t-s|$, Lemma~\ref{tight-lemma1} involves no arithmetic estimates and holds as an absolute inequality; the proof of Lemma~\ref{tight-lemma0} also holds essentially verbatim, only substituting \eqref{tilde-to-no-tilde} in place of \cite[(5)]{RicottaRoyerShparlinski2020} at the first step.

\begin{lemma}
\label{tight-lemma2}
For every $0<\lambda<1$, there exists a  $b=b(\lambda)>0$ such that, for every $\alpha\geqslant 2$ and $n\in\mathbb{N}$, if
\[ \frac1{\varphi(p^n)-1}\leqslant |t-s|\leqslant p^{-n\lambda}, \]
then
\[ \frac1{p^{n-1}}\sum_{a\in(\mathbb{Z}/p^n\mathbb{Z})^{\times}_{a_1}}\big|\Kl_{p^n}(t;(a,b_0))-\Kl_{p^n}(s;(a,b_0))\big|^{\alpha}\ll_{\alpha} |t-s|^{1+b(\alpha-2)}. \]
\end{lemma}

\begin{proof}
We start by observing that, from the definition of $\widetilde{\Kl_{p^n}}(\cdot;(a,b_0))$,
\begin{equation}
\label{Kl-Kl-incomplete-sum}
\widetilde{\Kl_{p^n}}(t;(a,b_0))-\widetilde{\Kl_{p^n}}(s;(a,b_0))=\frac1{p^{n/2}}\sum_{x\in\mathcal{I}_{s,t}}e_{p^n}(ax+b_0\bar{x}),
\end{equation}
for a certain interval $\mathcal{I}_{s,t}$ of length $|\mathcal{I}_{s,t}|\ll p^n|t-s|$; cf.~\cite[(19)]{RicottaRoyerShparlinski2020}.

Take for now an arbitrary $\alpha\geqslant 0$. If $\lambda\geqslant\frac23$, then $|t-s|\leqslant p^{-2n/3}$, and, estimating the exponential sum in \eqref{Kl-Kl-incomplete-sum} trivially and using Lemma~\ref{tight-lemma0}, we conclude that for every $a\in(\mathbb{Z}/p^n\mathbb{Z})^{\times}$,
\begin{equation}
\label{central-eq1}
\big|\Kl_{p^n}(t;(a,b_0))-\Kl_{p^n}(s;(a,b_0))\big|^{\alpha}
\ll_{\alpha} |t-s|^{\alpha/2}+|t-s|^{\alpha}p^{\alpha n/2}\ll_{\alpha} |t-s|^{\alpha/4}.
\end{equation}

In the more interesting case $\lambda<\frac23$, when $|t-s|>p^{-2n/3}$, estimating the exponential sum in \eqref{Kl-Kl-incomplete-sum} using Proposition~\ref{short-ish-sums}, and using Lemma~\ref{tight-lemma0}, we conclude that there exists a $\delta=\delta(\lambda)>0$ such that, for every $a\in(\mathbb{Z}/p^n\mathbb{Z})^{\times}$,
\begin{equation}
\label{central-eq2}
\begin{aligned}
\big|\Kl_{p^n}(t;(a,b_0))-\Kl_{p^n}(s;(a,b_0))\big|^{\alpha}
&\ll_{\alpha} |t-s|^{\alpha/2}+p^{-\alpha n\delta}\\
&\ll_{\alpha,\lambda} |t-s|^{\alpha/2}+|t-s|^{(3/2)\alpha\delta}.
\end{aligned}
\end{equation}

On the other hand, using orthogonality, we have that
\begin{equation}
\label{central-eq0}
\frac1{p^{n-1}}\sum_{a\in(\mathbb{Z}/p^n\mathbb{Z})^{\times}_{a_1}}\big|\widetilde{\Kl_{p^n}}(t;(a,b_0))-\widetilde{\Kl_{p^n}}(s;(a,b_0))\big|^2\ll\frac{|\mathcal{I}_{s,t}|}{p^n}\ll |t-s|.
\end{equation}
Combining \eqref{central-eq0} with \eqref{central-eq1} and \eqref{central-eq2} with $\alpha-2\geqslant 0$ in place of $\alpha$, we conclude that
\[ \frac1{p^{n-1}}\sum_{a\in(\mathbb{Z}/p^n\mathbb{Z})^{\times}_{a_1}}\big|\widetilde{\Kl_{p^n}}(t;(a,b_0))-\widetilde{\Kl_{p^n}}(s;(a,b_0))\big|^{\alpha}\ll_{\alpha}|t-s|^{1+b(\alpha-2)}, \]
with $b=\min(\frac14,\frac32\delta(\lambda))>0$. (The cut-off at $\frac23$ was chosen for concreteness, and the argument applies with a cutoff at any $\lambda_0>\frac12$, leading to slightly different choices for $0<b<\frac12$.)
\end{proof}

\begin{lemma}
\label{tight-lemma3}
For every even integer $\alpha\geqslant 2$, there exists a $0<\lambda<1$ such that, if
\[ p^{-\lambda n}\leqslant |t-s|\leqslant 1, \]
then
\[ \frac1{p^{n-1}}\sum_{a\in(\mathbb{Z}/p^n\mathbb{Z})^{\times}_{a_1}}\big|\Kl_{p^n}(t;(a,b_0))-\Kl_{p^n}(s;(a,b_0))\big|^{\alpha}\ll |t-s|^{\alpha/2}. \]
\end{lemma}

\begin{proof}
Analogously to the proof of \cite[Lemma 4.6]{RicottaRoyerShparlinski2020}, define a random variable
\[ \Kl_{p^n}\Big(t;\frac{p^n-1}2;(a_1,b_0)\Big)=\sum_{\substack{|h|\leqslant(p^n-1)/2\\ (a_1-h)b_0\in(\mathbb{Z}/p\mathbb{Z})^{\times 2}}}\alpha_{p^n}(h;t)U_h^{\sharp}, \]
where $U_h^{\sharp}$ is a sequence of independent random variables of probability law $\mu$ in \eqref{def-mu}, similarly to \eqref{KlpH-truncated} with $H=\frac12(p^n-1)$ but with coefficients $\alpha_{p^n}(h;t)$ in place of $\beta(h;t)$.
Then we have as in \cite[(20)]{RicottaRoyerShparlinski2020}
\[ \mathbb{E}\bigg(\bigg|\Kl_{p^n}\Big(t;\frac{p^n-1}2;(a_1,b_0)\Big)-\Kl_{p^n}\Big(s;\frac{p^n-1}2;(a_1,b_0)\Big)\bigg|^{\alpha}\bigg)\leqslant 32^{\alpha/2}c_{\alpha}|t-s|^{\alpha/2} \]
with a certain $c_{\alpha}>0$. This is an absolute, purely probabilistic inequality which does not depend on any estimates of Kloosterman sums.

On the other hand, expanding using \eqref{completion-eq1}, we obtain
\begin{align*}
&\frac{1}{p^{n-1}}\sum_{a\in(\mathbb Z/p^n\mathbb Z)^\times_{a_1}}|\widetilde{\Kl_{p^n}}(t;(a,b_0))-\widetilde{\Kl_{p^n}}(s;(a,b_0))|^{\alpha}\\
&\qquad=\frac{1}{p^{n-1}}\sum_{a\in(\mathbb Z/p^n\mathbb Z)^\times_{a_1}}\bigg|\frac1{p^{n/2}}\sum_{\substack{h\bmod p^n\\(a-h)b_0\in(\mathbb{Z}/p\mathbb{Z})^{\times 2}}}\big(\alpha_{p^n}(h;t)-\alpha_{p^n}(h;s)\big)\Kl_{p^n}(a-h;b_0)\bigg|^{\alpha}\\
&\qquad=\frac{1}{p^{\alpha n/2}}\sum_{\bm{h}\in H^{\alpha}_{a_1(p^n)}}\prod_{j=1}^{\alpha/2}\overline{\big(\alpha_{p^n}(h_j;t)-\alpha_{p^n}(h_j;s)\big)}\prod_{j=\alpha/2+1}^{\alpha}\big(\alpha_{p^n}(h_j;t)-\alpha_{p^n}(h_j;s)\big)\\
&\qquad\qquad\times\frac{1}{p^{n-1}}\sum_{a\in(\mathbb Z/p^n\mathbb Z)^\times_{a_1}}\prod_{j=1}^{\alpha}\Kl_{p^n}(a-h_j,b_0),
\end{align*}
as in \eqref{complex-moment-with-sums-of-products}. We now proceed as in the proofs of Propositions~\ref{main-error-terms-proposition} and \ref{error-term-estimate-proposition}, which includes the use of Theorem~\ref{sums-of-products-theorem} with a phase $f_{T,\bm{\epsilon}}$ as in \eqref{f-T-eps} with $|T|\leqslant\alpha$ and $\|\bm{\epsilon}\|_1\leqslant 2\alpha$. This shows (cf.~\eqref{main-term-eval} and Proposition~\ref{error-term-estimate-proposition}) that there exists a $\delta=\delta(\alpha)>0$ such that
\begin{align*}
&\frac1{p^{n-1}}\sum_{a\in(\mathbb{Z}/p^n\mathbb{Z})^{\times}_{a_1}}\big|\widetilde{\Kl_{p^n}}(t;(a,b_0))-\widetilde{\Kl_{p^n}}(s;(a,b_0))\big|^{\alpha}\\
&\qquad =\mathbb{E}\bigg(\bigg|\Kl_{p^n}\Big(t;\frac{p^n-1}2;(a_1,b_0)\Big)-\Kl_{p^n}\Big(s;\frac{p^n-1}2;(a_1,b_0)\Big)\bigg|^{\alpha}\bigg)+\mathrm{O}_{\alpha}\bigg(\frac1{p^{\delta n}}\bigg)\\
&\qquad\ll_{\alpha} |t-s|^{\alpha/2}+p^{-\delta n}.
\end{align*}
Picking for concreteness 
$\lambda=2\delta(\alpha)/\alpha$, and invoking Lemma~\ref{tight-lemma0} again, we conclude that for $|t-s|\geqslant p^{-\lambda n}$
\[ \frac1{p^{n-1}}\sum_{a\in(\mathbb{Z}/p^n\mathbb{Z})^{\times}_{a_1}}\big|\Kl_{p^n}(t;(a,b_0))-\Kl_{p^n}(s;(a,b_0))\big|^{\alpha}\ll_{\alpha} |t-s|^{\alpha/2}. \]
(Any choice $\lambda<\delta(\alpha)$ leads to an acceptable upper bound $\ll_{\alpha,\epsilon}|t-s|^{\alpha/2}+|t-s|^{1+\epsilon}$.)
\end{proof}

\subsection{Proof of Theorem~\ref{main}}
\label{proof-of-main}
Theorem~\ref{main} follows by combining Corollary~\ref{finite-distributions-corollary} and Corollary~\ref{tight-corollary} using Prokhorov's criterion (Proposition~\ref{prokhorov}).

\subsection{Rearranged paths}
\label{rearranged-paths-subsec}
In this section, we indicate changes in the proof of Theorem~\ref{main} that lead to a proof of Theorem~\ref{rearranged-paths-thm}.

According to the method of moments, as in \eqref{complex-moment-def} and Proposition~\ref{complex-moments-prop}, for convergence of $\Kl^{\circ}_{p^n}\to\Kl^{\circ}$ in the sense of finite distributions it suffices to show that for every two fixed $k$-tuples $\bm{m}=(m_1,\dots,m_k)$, $\bm{n}=(n_1,\dots,n_k)\in\mathbb{Z}_{\geqslant 0}^k$, there exists a $\delta=\delta(\|\bm{m}\|_1+\|\bm{n}\|_1)>0$ such that, for every $\bm{t}=(t_1,\dots,t_k)\in[0,1]^k$, $b_0\in(\mathbb{Z}/p\mathbb{Z})^{\times}$, and $n\geqslant 2$,
\begin{align*}
\mathcal{M}^{\circ}_{p^n}(\bm{t};\bm{m},\bm{n};b_0)
&:=\frac{1}{\varphi(p^n)/2}\sum_{a\in b_0(\mathbb Z/p^n\mathbb Z)^{\times 2}}\prod_{i=1}^k \overline{\Kl^{\circ}_{p^n}(t_i;(a,b_0))}^{m_i}\Kl^{\circ}_{p^n}(t_i;(a,b_0))^{n_i}\\
&=\mathbb E\Big(\prod_{i=1}^k\overline{\Kl^{\circ}(t_i)}^{m_i}\Kl^{\circ}(t_i)^{n_i}\Big)+\mathrm{O}_{\ell(\bm{m}+\bm{n})}(p^{-\delta n}).
\end{align*}
As in \eqref{modified-path} we define the renormalization $\widetilde{\Kl_{p^n}^{\circ}}(t;(a,b_0)):[0,1]\to\mathbb{C}$, for every $k\in\{1,\dots,p^{n-2}\}$ as
\[ \forall t\in\Big(\frac{k-1}{p^{n-2}},\frac{k}{p^{n-2}}\Big],\quad \widetilde{\Kl_{p^n}^{\circ}}(t;(a,b_0))=\frac{1}{p^{n/2}}\mathop{\sum\nolimits^\times}_{1\leqslant x\leqslant x_k^{\circ}(t)}f_{p^n}(x;(a,b)), \]
where $x_k^{\circ}(t)=\varphi(p^{n-1})t+k-1$. Completion corresponding to Lemma~\ref{completion-lemma} yields, with some careful normalization,
\begin{equation}
\label{completion-circ}
\widetilde{\Kl^{\circ}_{p^n}}(t;(a,b_0))
=\frac1{p^{n/2}}\sum_{h\bmod p^{n-1}}\alpha^{\circ}_{p^n}(h;t)\Kl_{p^n}(a-ph;b_0),
\end{equation}
where the coefficients
\[ \alpha^{\circ}_{p^n}(h;t)=\frac1{p^{n/2-1}}\sum_{1\leqslant x\leqslant x_k^{\circ}(t)}e_{p^{n-1}}(hx) \]
satisfy, for $|h|<\frac12p^{n-1}$, exactly the same estimates as $\alpha_{p^n}(h;t)$ in \eqref{completion-eq2}.
As in Lemma~\ref{moments-approx-lemma}, we prove that the corresponding complex moments $\widetilde{\mathcal{M}^{\circ}_{p^n}}(\bm{t};\bm{m};\bm{n};b_0)$ analogous to \eqref{complex-moment-renorm-def}
are within $\mathrm{O}_{\ell(\bm{m}+\bm{n})}(p^{-n/2}\log^{\ell(\bm{m}+\bm{n})}(p^n))$ of $\mathcal{M}^{\circ}_{p^n}(\bm{t};\bm{m};\bm{n};b_0)$;
on the other hand, substituting the completion expansion \eqref{completion-circ} we find analogously to \eqref{complex-moment-with-sums-of-products} that
\begin{align*}
\widetilde{\mathcal{M}^{\circ}_{p^n}}(\bm{t};\bm{m},\bm{n};b_0)
=\frac{1}{p^{n\ell(\bm{m}+\bm{n})/2}}&\sum_{\bm{h}\in H_{\circ}^{\ell(\bm{m}+\bm{n})}}\prod_{i=1}^k\prod_{j=1}^{m_i}\overline{\alpha^{\circ}_{p^n}(h_{i,j};t_i)}\prod_{j=m_i+1}^{m_i+n_i}\alpha^{\circ}_{p^n}(h_{i,j};t_i)\\
&\qquad\times\frac{1}{\varphi(p^n)/2}\sum_{a\in b_0(\mathbb Z/p^n\mathbb Z)^{\times 2}}\prod_{i=1}^k\prod_{j=1}^{m_i+n_i}\Kl_{p^n}(a-ph_{i,j},b_0),
\end{align*}
where $H_{\circ}=[-\frac12 p^{n-1},\frac12p^{n-1}]\cap\mathbb{Z}$. From here, the core argument proceeds in the same way; the main term arises from $\bm{h}\in H_{\circ}^{\ell(\bm{m}+\bm{n})}$ for which $\bm{\mu}_{\bm{h}}:\mathbb{Z}/p^{n-1}\mathbb{Z}\to\mathbb{Z}_{\geqslant 0}$ defined as in \S\ref{main-error-terms} satisfies $2\mid\bm{\mu}_{\bm{h}}$, and it evaluates as
\begin{equation}
\begin{aligned}
&\frac{1}{p^{n\ell(\bm{m}+\bm{n})/2}}\!\!\sum_{\bm{h}\in H^{\ell(\bm{m}+\bm{n})}_{\circ}}\prod_{i=1}^k\prod_{j=1}^{m_i}\overline{\alpha^{\circ}_{p^n}(h_{i,j};t_i)}\prod_{j=m_i+1}^{m_i+n_i}\!\!\alpha^{\circ}_{p^n}(h_{i,j};t_i)
\!\!\prod_{\tau\in T(\bm{\mu}_{\bm{h}})}\!\!\frac{\delta_{2\mid\bm{\mu}_{\bm{h}}(\tau)}}{2^{\bm{\mu}_{\bm{h}}(\tau)}}\binom{\bm{\mu}_{\bm{h}}(\tau)}{\bm{\mu}_{\bm{h}}(\tau)/2}\\
&\qquad=\sum_{\bm{h}\in H^{\ell(\bm{m}+\bm{n})}_{\circ}}\prod_{i=1}^k\prod_{j=1}^{m_i}\overline{\beta(h_{i,j};t_i)}\prod_{j=m_i+1}^{m_i+n_i}\beta(h_{i,j};t_i)\mathbb E\bigg(\prod_{i=1}^k\prod_{j=1}^{m_i+n_i}U_{h_{i,j}}^{\sharp}\bigg)\\
&\qquad\qquad\qquad+\mathrm{O}_{\ell(\bm{m}+\bm{n})}\bigg(\frac{\log^{\ell(\bm{m}+\bm{n})}p}{p^n}\bigg)\\
&\qquad=\mathbb E\bigg(\prod_{i=1}^k\overline{\Kl^{\circ}(t_i)}^{m_i}\Kl^{\circ}(t_i)^{n_i}\bigg)+\mathrm{O}_{\ell(\bm{m}+\bm{n})}\bigg(\frac{\log^{\ell(\bm{m}+\bm{n})}p}{p^{n/2}}\bigg),
\end{aligned}
\end{equation}
while the error term is $\mathrm{O}_{\ell(\bm{m}+\bm{n})}(p^{-\delta n})$ for some $\delta=\delta(\|\bm{m}\|_1+\|\bm{n}\|_1)>0$ as in \S\ref{error-terms-sec}.

Turning to the proof of convergence in law in \S\ref{tightness-thm1-subsec}, we prove the analogue of Proposition~\ref{tightness-prop}, namely that for every even integer $\alpha\geqslant 4$, there exists a $\beta=\beta(\alpha)>0$ such that, for every $0\leqslant s,t\leqslant 1$,
\[ \frac1{\varphi(p^n)/2}\sum_{a\in b_0(\mathbb{Z}/p^n\mathbb{Z})^{\times 2}}\big|\Kl^{\circ}_{p^n}(t;(a,b_0))-\Kl^{\circ}_{p^n}(s;(a,b_0))\big|^{\alpha}\ll_{\alpha} |t-s|^{1+\beta}, \]
which in particular shows that the sequence of random variables $\Kl_{p^n}^{\circ}$ is tight as $n\to\infty$. This proceeds analogously to \S\ref{tightness-thm1-subsec}, with the key arithmetic input being the exponential sum estimate analogous to \eqref{incomplete-kloost} that in intervals $\mathcal{I}$ of length $|\mathcal{I}|\leqslant p^{n(1-\lambda)}$,
\[ \frac1{p^{n/2}}\sum_{x\in\mathcal{I}}f_{p^n}(x;(a,b))=\frac{p}{p^{n/2}}\sum_{\substack{x\in\mathcal{I}\\x^2\equiv\bar{a}b\bmod p}}e_{p^n}(ax+b\bar{x})\ll_{\lambda}p^{-n\delta} \]
for some $\delta=\delta(\lambda)>0$, which in fact follows from \eqref{incomplete-kloost} or can be independently derived from \cite[Theorem 2]{Milicevic2016}. The proofs of Lemmata~\ref{tight-lemma1} and \ref{tight-lemma0} go through for $\Kl^{\circ}_{p^n}$ and $\widetilde{\Kl^{\circ}_{p^n}}$ with only trivial modifications, with the cutoff for $|t-s|$ replaced by $1/(\varphi(p^{n-1})-1)$. The same holds for the proofs of Lemmata~\ref{tight-lemma2} and \ref{tight-lemma3}, additionally replacing all averages over $a\in(\mathbb{Z}/p^n\mathbb{Z})^{\times}_{a_1}$ with averages over $a\in b_0(\mathbb{Z}/p^n\mathbb{Z})^{\times 2}$, and all sums over $|h|\leqslant \frac12(p^n-1)$ subject to $(a_1-h)b_0\in(\mathbb{Z}/p\mathbb{Z})^{\times 2}$ and corresponding shifted sums $\Kl_{p^n}(a-h;b_0)$ with sums over $|h|\leqslant \frac12(p^{n-1}-1)$ and corresponding shifted sums $\Kl_{p^n}(a-ph;b_0)$.

\subsection{Large \texorpdfstring{$p$}{p} limit}
\label{large-p-section}
In this section, we will prove the following theorem.
\begin{theorem}
\label{large-p-theorem}
Fix a nonzero integer $b_0$. Then the sequence of $C^0([0,1],\mathbb{C})$-valued random variables $\Kl^{\bullet}(\cdot;p)$ defined in \eqref{Kl-bullet} converges in law, as $p\to\infty$, to the random variable $\Kl$ defined in \eqref{KKl}.
\end{theorem}

Using Prokhorov's criterion for convergence in law of Proposition~\ref{prokhorov}, Theorem~\ref{large-p-theorem} follows directly from the following two statements.
\begin{proposition}
\label{fin-dis-large-p-prop}
Fix a nonzero integer $b_0$.
Then the sequence of $C^0([0,1],\mathbb{C})$-valued random variables $\Kl^{\bullet}(\cdot;p)$, defined for every odd prime $p$ in \eqref{Kl-bullet}, converges in the sense of finite distributions to the $C^0([0,1],\mathbb{C})$-valued random variable $\Kl$ defined in \eqref{KKl} as $p\to\infty$.
\end{proposition}
\begin{proposition}
\label{tight-large-p-prop}
Fix a nonzero integer $b_0$. Then the sequence of $C^0([0,1],\mathbb{C})$-valued random variables $\Kl^{\bullet}(\cdot;p)$ defined in \eqref{Kl-bullet} is tight over all odd primes $p$.
\end{proposition}

As in the case of Theorem~\ref{main}, the arithmetic heart of Theorem~\ref{large-p-theorem} is in Proposition~\ref{fin-dis-large-p-prop}, for which we import the key ingredient from \cite{RicottaRoyer2018}, adapting to our notation. We note that the proof of Proposition~\ref{RR-Weil} relies on Weil's version of Riemann Hypothesis for curves over finite fields.
\begin{proposition}[{\cite[Proposition 4.8]{RicottaRoyer2018}}]
\label{RR-Weil}
For a set $T\subseteq\mathbb{Z}/p^n\mathbb{Z}$ and $b_0\in(\mathbb{Z}/p^n\mathbb{Z})^{\times}$, let $(\mathbb{Z}/p^n\mathbb{Z})^{[T]}$ be as in \eqref{range-of-definition}, and let $\bar{T}=(T+p\mathbb{Z})/p\mathbb{Z}\subseteq\mathbb{Z}/p\mathbb{Z}$. Then
\[ \big|(\mathbb{Z}/p^n\mathbb{Z})^{[T]}\big|=\frac{\varphi(p^n)}{2^{|\bar{T}|}}\bigg(1+\mathrm{O}\bigg(\frac{2^{|\bar{T}|}|\bar{T}|}{p^{1/2}}\bigg)\bigg), \]
with the implied constant independent of $p$.
\end{proposition}

\begin{proof}[Proof of Proposition~\ref{fin-dis-large-p-prop}]
Similarly as in the proof of Theorem~\ref{main}, we may proceed by the method of moments. For this we find it convenient to use the truncated random variables  $\Kl^{\bullet}_{H}(\cdot;p)$ and $\Kl_H$ defined for $H>0$ analogously to \eqref{KlpH-truncated} as
\begin{equation}
\label{truncated-bullet}
\Kl^{\bullet}_{H}\Big|_{\Omega_{a_1}}(t;p)=\!\!\!\!\!\sum_{\substack{|h|\leqslant H\\(h-a_1)b_0\in(\mathbb{Z}/p\mathbb{Z})^{\times 2}}}\!\!\!\!\!\beta(h;t)U_{h,a_1}^{\sharp},\quad
\Kl_H(t)=\sum_{|h|\leqslant H}\beta(h;t)U_h,\quad (t\in [0,1]),
\end{equation}
where $\beta(h;t)$ is as in \eqref{completion-eq2}, $U_{h,a_1}^{\sharp}$ and $U_h$ are independent random variables of probability law $\mu$ in \eqref{def-mu} and $\mu_U$ in \eqref{KKl}, respectively, and keeping in mind the notation $\Omega=\bigsqcup_{a_1\in(\mathbb{Z}/p\mathbb{Z})^{\times}}\Omega_{a_1}$ for the underlying probability space of $\Kl^{\bullet}(\cdot;p)$ from Remark~\ref{all-a-remark}.

Fix $H>0$, $k\geqslant 1$, a $k$-tuple $\bm{t}=(t_1,\dots,t_k)\in[0,1]^k$, $\bm{m}=(m_1,\dots,m_k)$, and $\bm{n}=(n_1,\dots,n_k)\in\mathbb{Z}_{\geqslant 0}^k$, denote $\ell(\bm{m}+\bm{n})=\sum_{i=1}^k(m_i+n_i)$, and consider the corresponding complex moment of $\Kl^{\bullet}(\cdot;p)$:
\begin{align*}
\mathcal{M}^{\bullet}_{H}(\bm{t};p;\bm{m};\bm{n};b_0)
&=\frac1{p-1}\sum_{a_1\in(\mathbb{Z}/p\mathbb{Z})^{\times}}\mathbb{E}\Big(\prod_{i=1}^k\overline{\Kl^{\bullet}_{H}}(t_i;p)\Big|_{\Omega_{a_1}}^{m_i}\Kl^{\bullet}_{H}(t_i;p)\Big|_{\Omega_{a_1}}^{n_i}\Big)\\
&=\frac1{p-1}\sum_{a_1\in(\mathbb{Z}/p\mathbb{Z})^{\times}}\!\sum_{\bm{h}\in [-H,H]^{\ell(\bm{m}+\bm{n})}}
\!\!\!\!\!\bm{\delta}(\bm{h};(a_1,b_0))\bm{\beta}(\bm{h};\bm{t})\mathbb{E}\Big(\prod_{i=1}^k\prod_{j=1}^{m_i+n_i}U^{\sharp}_{h_{i,j},a_1}\Big)\\
&=\sum_{\bm{h}\in [-H,H]^{\ell(\bm{m}+\bm{n})}}\frac{\big|(\mathbb{Z}/p\mathbb{Z})^{[T_{p,\bm{h}}]}\big|}{p-1}\bm{\beta}(\bm{h};\bm{t})\mathbb{E}\Big(\prod_{i=1}^k\prod_{j=1}^{m_i+n_i}U^{\sharp}_{h_{i,j}}\Big),
\end{align*}
where $\bm{h}=(\bm{h}_1,\dots,\bm{h}_k)\in [-H,H]^{\ell(\bm{h}+\bm{n})}$, $\bm{h}_i=(h_{i,1},\dots,h_{i,m_i},h_{i,m_i+1},\dots,h_{i,m_i+n_i})\in [-H,H]^{m_i+n_i}$,
\[ \bm{\delta}(\bm{h};(a_1,b_0))=\prod_{i=1}^k\prod_{j=1}^{m_i+n_i}\delta_{(a_1-h_{i,j})b_0\in(\mathbb{Z}/p\mathbb{Z})^{\times 2}},\,\, \bm{\beta}(\bm{h};\bm{t})=\prod_{i=1}^k\prod_{j=1}^{m_i}\overline{\beta(h_{i,j};t_i)}\prod_{j=m_i+1}^{m_i+n_i}\beta(h_{i,j};t_i), \]
$(U_h^{\sharp})$ is another sequence of independent random variables of probability law $\mu$, and
\[ T_{p,\bm{h}}=(T_{\bm{h}}+p\mathbb{Z})/p\mathbb{Z},\quad T_{\bm{h}}=\big\{h_{i,j}:1\leqslant i\leqslant k,\,1\leqslant j\leqslant m_i+n_i\big\}. \]
Note that $|T_{\bm{h},p}|=|T_{\bm{h}}|$ for sufficiently large $p$ (namely for $p>2H$).
Applying Proposition~\ref{RR-Weil}, we have that for $p>2H$
\begin{equation}
\label{MpH-MH}
\begin{aligned}
&\mathcal{M}^{\bullet}_{H}(\bm{t};p;\bm{m};\bm{n};b_0)\\
&\qquad=\sum_{\bm{h}\in [-H,H]^{\ell(\bm{m}+\bm{n})}}
\bm{\beta}(\bm{h};\bm{t})
\frac1{2^{|T_{\bm{h}}|}}\mathbb{E}\Big(\prod_{i=1}^k\prod_{j=1}^{m_i+n_i}U^{\sharp}_{h_{i,j}}\Big)\Big(1+\mathrm{O}\Big(\frac{|T_{\bm{h}}|2^{|T_{\bm{h}}|}}{p^{1/2}}\Big)\Big)\\
&\qquad=\sum_{\bm{h}\in [-H,H]^{\ell(\bm{m}+\bm{n})}}
\bm{\beta}(\bm{h};\bm{t})\mathbb{E}\Big(\prod_{i=1}^k\prod_{j=1}^{m_i+n_i}U_{h_{i,j}}\Big)+\mathrm{O}_{\ell(\bm{m}+\bm{n})}\Big(\frac{\log^{\ell(\bm{m}+\bm{n})}H}{p^{1/2}}\Big)\\
&\qquad=\mathcal{M}_H(\bm{t};\bm{m};\bm{n})+\mathrm{O}_{\ell(\bm{m}+\bm{n})}\Big(\frac{\log^{\ell(\bm{m}+\bm{n})}H}{p^{1/2}}\Big),
\end{aligned}
\end{equation}
where $\mathcal{M}_H(\bm{t};\bm{m};\bm{n})$ is the corresponding complex moment of $\Kl_H$. Taking $p\to\infty$, we conclude that, for every fixed $H>0$, the $C^0([0,1],\mathbb{C})$-valued random variable $\Kl^{\bullet}_{H}(\cdot;p)$ converges, in the sense of finite distributions, to the random variable $\Kl_H$.

On the other hand, the analogue of Proposition~\ref{random-series-properties} holds for the random variables $\Kl^{\bullet}(\cdot;p)$ and $\Kl$, simply by taking a convex linear combination of $\Kl(t;p;(a_1,b_0))$ in the former case and by \cite[Proposition 3.1]{RicottaRoyer2018} in the latter. As in \eqref{main-term-eval}, this implies that
\begin{equation}
\label{MpH-Mp}
\begin{aligned}
&\mathcal{M}^{\bullet}_{H}(\bm{t};p;\bm{m};\bm{n};b_0)-\mathcal{M}^{\bullet}(\bm{t};p;\bm{m};\bm{n};b_0)\\
&\qquad=\mathbb{E}\Big(\prod_{i=1}^k\overline{\Kl^{\bullet}_{H}}(t_i;p)^{m_i}\Kl^{\bullet}_{H}(t_i;p)^{n_i}-\prod_{i=1}^k\overline{\Kl^{\bullet}}(t_i;p)^{m_i}\Kl^{\bullet}(t_i;p)^{n_i}\Big)\\
&\qquad\ll_{\ell(\bm{m}+\bm{n})}\frac{(\log H)^{\ell(\bm{m}+\bm{n})}}{H^{1/2}},
\end{aligned}
\end{equation}
and similarly $\mathcal{M}_H(t;\bm{m};\bm{n})-\mathcal{M}(t;\bm{m};\bm{n})\ll_{\ell(\bm{m}+\bm{n})}(\log H)^{\ell(\bm{m}+\bm{n})}/H^{1/2}$, both uniformly in $p$. From \eqref{MpH-MH} and \eqref{MpH-Mp}, choosing for example $H=\frac12(p-1)$ it follows that
\[ \mathcal{M}^{\bullet}(\bm{t};p;\bm{m};\bm{n};b_0)=\mathcal{M}(\bm{t};\bm{m};\bm{n})+\mathrm{O}\Big(\frac{\log^{\ell(\bm{m}+\bm{n})}p}{p^{1/2}}\Big). \]
From this it follows that the sequence of $C^0([0,1],\mathbb{C})$-valued random variables $\Kl^{\bullet}(\cdot;p)$ converges in the sense of finite distributions to the random variable $\Kl$ as $p\to\infty$, as announced.
\end{proof}

\begin{proof}[Proof of Proposition~\ref{tight-large-p-prop}]
We adapt to the present situation the first part of the proof of \cite[Lemma 4.6]{RicottaRoyerShparlinski2020}, starting from the truncated random variable $\Kl^{\bullet}_{H}(\cdot;p)$ defined for every $H>0$ in \eqref{truncated-bullet}. Since the random variables $U^{\sharp}_h$ are independent, and each of them is 4-subgaussian~\cite[Proposition B.8.2]{KowalskiBook}, we have for every $u\in\mathbb{R}$
\begin{equation}
\label{bullet-subgaussian}
\begin{aligned}
\mathbb{E}\Big(e^{u(\Kl^{\bullet}_{H}(t;p)-\Kl^{\bullet}_{H}(s;p))}\Big)
&=\frac1{p-1}\sum_{a_1\in(\mathbb{Z}/p\mathbb{Z})^{\times}}\prod_{\substack{|h|\leqslant H\\(a_1-h)b_0\in(\mathbb{Z}/p\mathbb{Z})^{\times 2}}}\mathbb{E}\Big(e^{u(\beta(h;t)-\beta(h;s))U^{\sharp}_h}\Big)\\
&\leqslant\prod_{|h|\leqslant H}e^{4|\beta(h;t)-\beta(h;s)|^2u^2/2}=e^{(4\sigma_H^2)u^2/2},
\end{aligned}
\end{equation}
where, by the Plancherel formula,
\[ \sigma_H^2=4\sum_{|h|\leqslant H}|\beta(h;t)-\beta(h;s)|^2\leqslant 4\sum_{h\in\mathbb{Z}}\big|\widehat{\mathbf{1}_{[t,s]}}(h)\big|^2\leqslant 4\|\mathbf{1}_{[t,s]}\big\|_2^2=4|t-s|. \]
Since by definition \eqref{bullet-subgaussian} shows that the random variable $\Kl^{\bullet}_{H}(t;p)-\Kl^{\bullet}_{H}(s;p))$ is $\sigma_H$-subgaussian, we have by \cite[Proposition B.8.3]{KowalskiBook} for every $\alpha\in\mathbb{Z}_{\geqslant 0}$ 
\begin{equation}
\label{tight-truncated}
\mathbb{E}\big(\big|\Kl^{\bullet}_{H}(t;p)-\Kl^{\bullet}_{H}(s;p))\big|^{\alpha}\big)\leqslant c_{\alpha}\sigma_H^{\alpha}\leqslant 2^{\alpha}c_{\alpha}|t-s|^{\alpha/2}.
\end{equation}

On the other hand, for an even integer $\alpha$, it follows from Proposition~\ref{random-series-properties} as in \eqref{main-term-eval} that
\[ \mathbb{E}\big(\big|\Kl^{\bullet}_{H}(t;p)-\Kl^{\bullet}_{H}(s;p))\big|^{\alpha}\big)=\mathbb{E}\big(\big|\Kl^{\bullet}(t;p)-\Kl^{\bullet}(s;p))\big|^{\alpha}\big)+\mathrm{O}_{\alpha}\Big(\frac{\log^{\alpha}H}{H^{1/2}}\Big), \]
uniformly in $p$. We may thus take limits as $H\to\infty$ in \eqref{tight-truncated}, yielding
\[ \mathbb{E}\big(\big|\Kl^{\bullet}(t;p)-\Kl^{\bullet}(s;p))\big|^{\alpha}\big)\leqslant 2^{\alpha}c_{\alpha}|t-s|^{\alpha/2}. \]
Taking any even integer $\alpha\geqslant 4$, we see that the sequence $\Kl^{\bullet}(\cdot;p)$ satisfies Kolmogorov's criterion for tightness and is therefore tight by Proposition~\ref{kolmogorov}.
\end{proof}

\bibliographystyle{amsalpha}
\bibliography{KloostPaths} 

\end{document}